\documentclass[12pt,reqno]{amsart}

\usepackage{amsmath,amsthm,amssymb,comment,fullpage}
\usepackage{times}
\usepackage[T1]{fontenc}
\usepackage{mathrsfs}
\usepackage{latexsym}
\usepackage[dvips]{graphics}
\usepackage{epsfig}
\usepackage{accents}
\usepackage{float}
\usepackage{amsmath,amsfonts,amsthm,amssymb,amscd}
\input amssym.def
\input amssym.tex
\usepackage{color}
\usepackage{hyperref}
\usepackage{url}
\usepackage{breakurl}
\usepackage{comment}
\usepackage{mathtools}
\usepackage{thmtools}
\usepackage{thm-restate}
\newcommand{\bburl}[1]{\textcolor{blue}{\url{#1}}}

\newcommand{\E}{\mathbb{E}}

\renewcommand{\E}{\mathbb{E}}

\numberwithin{equation}{section}

\newtheorem{thm}{Theorem}[section]

\newtheorem{lem}[thm]{Lemma}

\newtheorem{exa}[thm]{Example}
\newtheorem{defi}[thm]{Definition}

\theoremstyle{plain}

\newtheorem{lemma}[thm]{Lemma}
\newtheorem{proposition}[thm]{Proposition}
\newtheorem{theorem}[thm]{Theorem}

\newtheorem{rem}[thm]{Remark}

\newtheorem{rek}[thm]{Remark}

\newcommand\be{\begin{equation}}
\newcommand\ee{\end{equation}}
\newcommand\bea{\begin{eqnarray}}
\newcommand\eea{\end{eqnarray}}
\newcommand\bi{\begin{itemize}}
	\newcommand\ei{\end{itemize}}
\newcommand\ben{\begin{enumerate}}
	\newcommand\een{\end{enumerate}}
\newcommand\bc{\begin{center}}
	\newcommand\ec{\end{center}}
\newcommand\ba{\begin{array}}
	\newcommand\ea{\end{array}}

\newcommand{\ga}{\alpha}

\newcommand{\R}{\ensuremath{\mathbb{R}}}
\newcommand{\C}{\ensuremath{\mathbb{C}}}
\newcommand{\Z}{\ensuremath{\mathbb{Z}}}

\newcommand{\N}{\mathbb{N}}

\newcommand{\hr}[1]{\href{#1}{\url{#1}}}

\renewcommand \l {\lambda}

\newcommand{\var}{\text{Var}}

\newcommand{\Tr}{\text{Tr}}
\DeclareMathOperator{\tr}{Tr}

\newcommand*{\reff}[1]{\hyperref[#1]{\ref{#1}}}

\DeclareMathOperator{\rank}{rank}

\title{The limiting spectral measure for an ensemble of generalized checkerboard matrices}

\author{Fangu Chen}
\email{\textcolor{blue}{\href{mailto:fangu@umich.edu}{fangu@umich.edu}}}
\address{Department of Mathematics, University of Michigan, Ann Arbor, MI 48109}

\author{Yuxin Lin}
\email{\textcolor{blue}{\href{mailto:ylin9@nd.edu}{ylin9@nd.edu}}}
\address{Department of Mathematics, University of Notre Dame, Notre Dame, IN 46556}

\author{Steven J. Miller}
\email{\textcolor{blue}{\href{mailto:sjm1@williams.edu, Steven.Miller.MC.96@aya.yale.edu}{sjm1@williams.edu,Steven.Miller.MC.96@aya.yale.edu}}}
\address{Department of Mathematics and Statistics, Williams College, Williamstown, MA 01267}

\author{Jiahui Yu}
\email{\textcolor{blue}{\href{jyad2018@mymail.pomona.edu}{jyad2018@mymail.pomona.edu}}}
\address{Department of Mathematics, Pomona College, Claremont, CA 91711}

\thanks{The authors were partially supported by NSF Grants DMS1947438 and  DMS1561945, the University of Michigan, the University of Notre Dame, Pomona College and Williams College. We thank the authors of \cite{BCDHMSTPY}, especially Roger Van Peski, for some helpful conversations and Akihiro Takigawa for the help on programming.}



\date{\today}

\begin{document}

\begin{abstract} 
Random matrix theory successfully models many systems, from the energy levels of heavy nuclei to zeros of $L$-functions. While most ensembles studied have continuous spectral distribution, Burkhardt et al introduced the ensemble of $k$-checkerboard matrices, a variation of Wigner matrices with entries in generalized checkerboard patterns fixed and constant. In this family, $N-k$ of the eigenvalues are of size $O(\sqrt{N})$ and were called bulk while the rest are tightly contrained around a multiple of $N$ and were called blip.

We extend their work by allowing the fixed entries to take different constant values. We can construct ensembles with blip eigenvalues at any multiples of $N$ we want with any multiplicity (thus we can have the blips occur at sequences such as the primes or the Fibonaccis). The presence of multiple blips creates technical challenges to separate them and to look at only one blip at a time. We overcome this by choosing a suitable weight
function which allows us to localize at each blip, and then exploiting cancellation to deal with the resulting combinatorics to determine the average moments of the ensemble; we then apply standard methods from probability to prove that almost surely the limiting distributions of the matrices converge to the average behavior as the matrix size tends to infinity. For blips with just one eigenvalue in the limit we have convergence to a Dirac delta spike, while if there are $k$ eigenvalues in a blip we again obtain hollow $k \times k$ GOE behavior.
\end{abstract}

\maketitle
\tableofcontents


\section{Introduction}



\subsection{Background}\label{sec:background}
Initially introduced by Wishart \cite{Wis} for some problems in statistics, random matrix theory has successfully modeled a large number of systems from energy levels of heavy nuclei to zeros of the Riemann zeta function; see for example the surveys \cite{Bai, BFMT-B, Con, FM, KaSa, KeSn} and the textbooks \cite{Fo, Meh, MT-B, Tao2}.
A simple but important example is the ensemble of real symmetric matrices whose upper triangular entries are independent, identically distributed random variables from some fixed probability distribution with mean $0$, variance $1$ and finite higher moments. Wigner's semi-circle law states that as the size of the matrix $N\to\infty$, the properly normalized spectral distribution of a matrix from the ensemble converges almost surely to a semi-circle (or semi-ellipse):
\begin{equation}\label{wignersemi}
    \sigma_{R}(x) = \begin{cases}
\frac{2}{\pi R^2}\sqrt{R^2 - x^{2}} & \text{{\rm if} } |x| \leq R,\\
0 & \text{{\rm if} } |x| > R.
\end{cases}
\end{equation}
See \cite{Wig1, Wig2, Wig3, Wig4, Wig5} for more details.


Besides the more well-known families such as the Gaussian Orthogonal, Unitary and Symplectic Ensembles, many other special ensembles have been studied; see for example \cite{Bai, BasBo1, BasBo2, BanBo, BLMST, BCG, BHS1, BHS2, BM, BDJ, GKMN, HM, JMRR, JMP, Kar, KKMSX, LW, MMS, MNS, MSTW, McK, Me, Sch}, where the additional structures on the entries of the matrices lead to different behaviors of the eigenvalues in the limit.



For most ensembles that people have studied, while it is possible to prove the convergence of the limiting spectral measure, in only a few (such as $d$-regular graphs \cite{McK}, block circulant matrices \cite{KKMSX} and palindromic Toeplitz matrices \cite{MMS}) can the limiting distribution be written down in a nice, closed form expression.


This paper is a sequel to \cite{BCDHMSTPY}, where they introduce ensembles of checkerboard matrices which also have a nice, closed-form expression for its limiting distribution. The spectrum splits into two; most of the eigenvalues are in the bulk and are of size $\sqrt{N}$,  but a small number are of size $N$. They studied the splitting behavior of the ensemble similar to that in \cite{CDF1, CDF2}, and used the combinatorial method in the style of \cite{KKMSX, MMS}. The ensemble in \cite{BCDHMSTPY} is defined as follows in the real symmetric case.  \begin{defi}\label{defn:checkerboard}
Fix $k\in\mathbb{N}$ and $w\in\mathbb{R}$. The $N \times N$ $(k,w)$-checkerboard ensemble over $\mathbb{R}$ is the ensemble of matrices $M = (m_{ij})$ given by
\begin{equation}\label{def $k$-checkerboard matrix}
m_{ij}\ =\ \begin{cases}
a_{ij} & \text{{\rm if} } i \not\equiv j \bmod k\\
w & \text{{\rm if} } i \equiv j \bmod k,
\end{cases}
\end{equation}
where $a_{ij}=a_{ji}$ are i.i.d. random variables with mean 0, variance 1, and finite higher moments, and the probability measure on the ensemble given by the natural product probability measure.
\end{defi} 

For this ensemble $N-k$ of the eigenvalues (called the \textbf{bulk} eigenvalues) are of order $\sqrt{N}$ and converge to a semi-circle, while $k$ of the eigenvalues (called the \textbf{blip} eigenvalues) are of order $N$ and converge to the spectral distribution of a $k\times k$ hollow Gaussian orthogonal ensemble. 


\begin{defi} \label{def hollow GOE}
The $k\times k$ \textbf{hollow Gaussian Orthogonal Ensemble} is given by $k\times k$ matrices $A = (a_{ij}) = A^T$ with
\begin{equation}
a_{ij} = \begin{cases}
\mathcal{N}_{\R}(0, 1) & \text{{\rm if} } i \neq j \\
0 & \text{{\rm if} } i = j.
\end{cases}
\end{equation}
\end{defi}
See \cite{BCDHMSTPY} for a collection of histograms of eigenvalues of matrices from some $k\times k$ hollow GOE.
\subsection{Generalized Checkerboard Ensembles}
We generalize \cite{BCDHMSTPY} by allowing the constant $w$ to take different values. 
While the Checkerboard ensembles in \cite{BCDHMSTPY} only allow one blip for each ensmble, the generalized Checkerboard ensembles allow arbitrarily many blips for each ensemble. Moreover, we have control over the positions of these blips.
That is, given a list of points, the generalized checkerboard ensemble allows the spectrum at those points in a ``non-trivial'' way. We can 
always ``trivially'' construct ensembles with prescribed locations and frequency by taking a diagonal union of block matrices. But then the blocks are independent from each other. The significance of the generalized checkerboard ensemble is that we can control the locations of normalized eigenvalues within an ensemble that doesn't have independent diagonal blocks. It is a "mixed" matrix whose eigenvalues have a nice split limiting distribution.


\begin{defi}\label{generalized cb}
Fix $k\in \mathbb{N}$ and a $k$-tuple of real numbers $W = (w_1,\dots, w_{k})$, then the $N\times N$ $(k, W)$-checkerboard ensemble is the ensemble of matrices $A_N=(m_{ij})$ given by \begin{align}
    m_{ij} \; = \; \begin{cases}
    a_{ij} & \text{ if } i\not\equiv j\pmod{k}, \\
    w_{u} & \text{ if } i\equiv j \equiv u\pmod{k},\text{ with  } u\in \{1,2,\dots, k\},
    \end{cases}
\end{align} where $a_{ij} = a_{ji}$ are independent and identically distributed random variables with mean $0$, variance $1$, and finite higher moments.
\end{defi}



For example, when $k=3$, $W=(1,1,2)$, a $ (3,W)$ checkerboard looks like the following (we assume $3|N$):

$$X = \begin{pmatrix}
	1 & a_{12} & a_{13} & 1 & a_{15} & a_{16} & \dots & 1 & a_{1N-1} & a_{1N} \\
	a_{12} & 1 & a_{23} & a_{24} & 1 & a_{26} & \dots & a_{2N-2} & 1 &a_{2N} \\
	a_{13} & a_{23} & 2 & a_{34}  & a_{35} & 2 &\dots & a_{3 N-2} & a_{3N-1} & 2 \\
	\vdots & \vdots & \vdots & \vdots & \vdots & \vdots & \ddots & \vdots & \vdots & \vdots \\
     a_{1N} & a_{2N} & 2 & a_{4N} & a_{5N} & 2 & \dots & a_{N-2N} & a_{N-1N} & 2 \\
	\end{pmatrix}.$$
	
\subsection{Results}
What makes the checkerboard ensemble in \cite{BCDHMSTPY} interesting is that the eigenvalues of a matrix from the ensemble almost surely fall into two separate regimes. With our generalization we can exploit the freedom to choose different constants to force the eigenvalues to fall into more regimes. To be more precise, using matrix perturbation theory we prove the following result.

\begin{restatable}{theorem}{bulkregimes}
Let $\{A_N\}_{N\in\N}$ be a sequence of $(k,W)$-checkerboard matrices. Suppose that $W$ has $x$ non-zero entries and there are $s$ distinct $w$'s, then almost surely as $N\rightarrow \infty$, the eigenvalues of $A_N$ fall into $s+1$ regimes: $N-x$ of the eigenvalues are $O(N^{1/2+\epsilon})$ and if $w_i'$ appears $k_i$ times, $k_i$ eigenvalues are of magnitude $Nw_i'/k+O(N^{1/2+\epsilon})$.
\end{restatable}

As in \cite{BCDHMSTPY}, we refer to the $N-x$ eigenvalues that are on the order of $\sqrt{N}$ as the eigenvalues in the \textbf{bulk}, while for each distinct $w_i$, the $k_i$ eigenvalues near $Nw_i/k$ are called the eigenvalues in the \textbf{blips}. We study the eigenvalue distribution of each regime. 


For the remainder of this paper,  $\boldsymbol{A_N}$ always refers to an $N \times N$ matrix.


Let $\nu_{A_N}$ be the empirical spectral measure of an $N\times N$ matrix $A_N$, where we have normalized the eigenvalues by dividing by $\sqrt{N}$:
\begin{equation}\label{eqn Wigner Spectral Distribution Measure}
\nu_{A_ N}(x)\ =\ \frac{1}{N}\sum_{\lambda \text{ \rm an eigenvalue of } A_N} \delta\left(x - \frac{\lambda}{\sqrt{N}}\right).
\end{equation}
For example, Figure \ref{figexampleoftheblip} gives this normalized eigenvalue distribution of a collection of $500\times 500$ $(6, W)$-checkerboard matrices with $W = (1,-2,-2,3,3,3)$.
\begin{figure}[h]
\begin{center}
\includegraphics[width = 16cm, height = 7 cm]{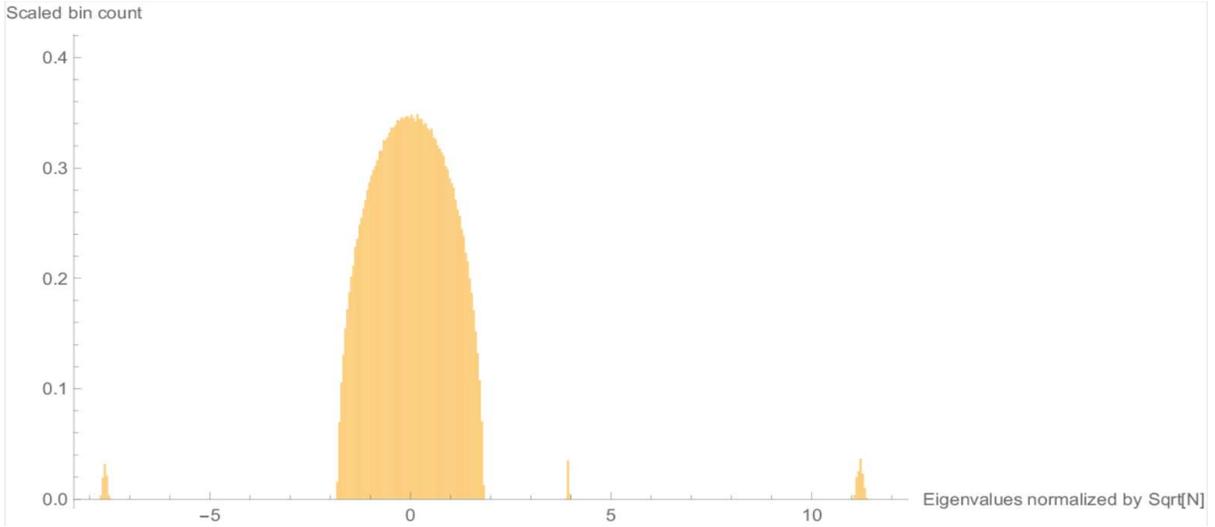}
\caption{\label{figexampleoftheblip} A histogram of the normalized eigenvalue distribution on a probability density scale for $500 \times 500$ $(6, W)$-checkerboard real matrices with $W = (1,-2,-2,3,3,3)$ after 500 trials.}
\end{center}
\end{figure}

As in \cite{BCDHMSTPY}, the blip eigenvalues of order $N$ prevent us from directly using the method of moments (for large $m$, the contribution from these eigenvalues dwarfs that from the bulk). We use following result (see \cite{Tao}) to bypass the complications presented by the small number of blip eigenvalues.

\begin{restatable}{prop}{stabilityESDrank}
\label{stabilityESDrank}
(\cite{Tao}) Let $\{\mathcal{A}_N\}_{N \in \N}$ be a sequence of random Hermitian matrix ensembles such that $\{\nu_{\mathcal{A}_N,N}\}_{N \in \N}$ converges weakly almost surely to a limit $\nu$. Let $\{\tilde{\mathcal{A}}_N\}_{N \in \N}$ be another sequence of random matrix ensembles such that $\frac{1}{N}\rank(\tilde{\mathcal{A}}_N)$ converges almost surely to zero. Then $\{\nu_{\mathcal{A}_N+\tilde{\mathcal{A}}_N,N}\}_{N\in\N}$ converges weakly almost surely to $\nu$.
\end{restatable}

Taking $\tilde{\mathcal{A}}_N$ to be the fixed matrix with entries $m_{ij}=w_u$ whenever $i\equiv j \equiv u \pmod{k}$ and zero otherwise, we have that the limiting spectral distribution of the $(k,W)$-checkerboard ensemble is the same as the limiting spectral distribution of the ensemble with $W=\mathbf{0}$, which does not have the $k$ large blip eigenvalues 
. This overcomes the issue of diverging moments.

\begin{restatable}{theorem}{thmbulklim}\label{thm_main_bulk}
Let $\{A_N\}_{N \in \N}$ be a sequence of $N \times N$ $(k,W)$-checkerboard matrices, and let $\nu_{A_N}$ denote the empirical spectral measure, then $\nu_{A_N}$ converges weakly almost surely to the Wigner semicircle measure $\sigma_R$ with radius
\begin{equation}
R\ =\ 2\sqrt{1-1/k}.
\end{equation}
\end{restatable}

The proof is by standard combinatorial arguments. We give the details in \S\ref{bulk}.

Similar to the previous checkerboard paper \cite{BCDHMSTPY}, each blip may be thought of as deviations about the trivial eigenvalues. Instead of having just one blip as in \cite{BCDHMSTPY}, we now have many different blips. A blip containing $k_i > 1$ eigenvalues has the same distribution as the eigenvalues of the $k_i \times k_i$ hollow Gaussian Orthogonal Ensemble (see Definition \ref{def hollow GOE}); when $k_i=1$ the blip has the distribution of a dirac delta function.

We need to define a weighted blip spectral measure which takes into account only the eigenvalues of one blip. Thus we not only need to get rid of the interference from the bulk, we also need to avoid the interference from the other blips. In order to facilitate the use of eigenvalue trace lemma,  similar to \cite{BCDHMSTPY}, we are led to use a polynomial weighting function and we use a sequence of polynomials of degree tending to infinity as the matrix size $N \rightarrow \infty$ so that in the limit we mimic a smooth cutoff function. Specifically, let
\begin{equation}
f_i^{2n}(x) \;:=\;\left(\frac{x(2-x)\prod_{w_{j}\neq w_i}(x-\frac{w_j}{w_i})(2-x-\frac{w_j}{w_i})}{\prod_{w_{j}\neq w_i}(1-\frac{w_j}{w_i})^2}\right)^{2n}.
\end{equation}
Thus we alter the standard empirical spectral measure in the following way to capture the blip. 
For example, when $k = 3$ and $W = (1,1,2)$, we use the polynomial $f_3^{2n}(x) = \left(x(2-x)(2x-1)(3-2x)\right)^{2n}$ to study the blip around $\frac{2N}{3}$. Figure \ref{fig:weightfn} gives a plot of the polynomial $f_3^{200}(x)$. We can see that the weight function $f_3^{2n}(x)$ is large when $|x-1| > \frac{\sqrt{5}}{2}$, but this would not cause a problem since almost surely there will be no scaled eigenvalues in that region.
\begin{figure}[H]
\begin{center}
\includegraphics[width = 9cm, height = 5 cm]{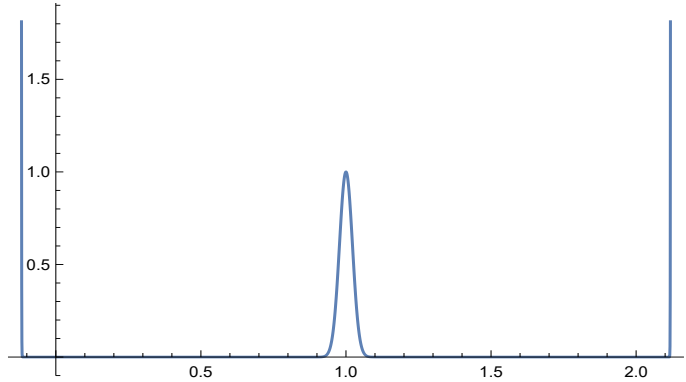}
\caption{\label{fig:weightfn} Weight function with $n = 100$ to study the blip around $\frac{2N}{3}$ when $k=3$ and $ W = (1,1,2)$.}
\end{center}
\end{figure}

\begin{defi}\label{def:blipmeasi}
Given $k\in\mathbb{N}$ and a $k$-tuple of real numbers $W = (w_1,\dots, w_k)$, the \textbf{empirical blip spectral measure} associated to an $N \times N$ $(k,W)$-checkerboard matrix $A_N$ around $Nw_i/k\neq 0$ is
\begin{equation}\label{eqn:blipmeasold}
\mu_{A_N,i}(x) \; :=\;\frac{1}{k_i} \sum_{\lambda \text{ \rm an eigenvalue of } A}  f_i^{2n}\left(\frac{k\lambda}{w_i N}\right)\delta\left(x-\left(\lambda-\frac{w_iN}{k}\right)\right),
\end{equation} where $k_i$ is the number of $w_i$'s in $(w_1,\dots, w_k)$, and $n(N)$ is a function satisfying $\lim\limits_{N\to\infty} n(N) = \infty$ and $n(N) = O(\log \log N)$.
\end{defi}

\begin{rek}\label{rmk:long_blip_justification}
The actual choice of weight functions should not change the empirical blip spectral measure in the limit. It will be used in the proof that the weight polynomial $f_i^{2n}(x)$ has a critical point at $1$ with $f_i^{2n}(1) = 1$ and has zeroes of order $2n$ at $0$ and at all $w_j/w_1$ with $w_j\neq w_1$.
Heuristically, because the fluctuation of eigenvalues in each regime is of order $\sqrt{N}$, we have $f_i^{2n}\left(\frac{k\lambda}{w_i N}\right) \approx 1$ if $\lambda$ is in the blip around $Nw_i/k$, and $f_i^{2n}\left(\frac{k\lambda}{w_i N}\right) \approx 0$ if $\lambda$ is in the bulk or in the blip other than $Nw_i/k$. More specifically, \begin{align}
    f_i^{2n}\left(\frac{k\lambda}{w_i N}\right) \; = \; \begin{cases}
    O\left(\frac{\log{N}}{N^n}\right) & \text{ {\rm if} } \lambda \text{ {\rm is} } O\left(\sqrt{N}\right) \text{ {\rm or} }  \frac{Nw_j}{k} +  O\left(\sqrt{N}\right) \text{ {\rm with} } w_j\neq w_i, \\
    1 + O\left(\frac{\log{N}}{N^{2n}}\right) & \text{ {\rm if} }  \lambda \text{ {\rm is} } \frac{Nw_i}{k} +  O\left(\sqrt{N}\right).
    \end{cases}
\end{align}
\end{rek}

As in \cite{BCDHMSTPY}, we use the method of moments to reduce to a combinatorial problem and relate the expected moments of the empirical blip measure around $Nw_i/k$ to those of the $k_i\times k_i$ hollow GOE. One remarkable observation is that the values of the constants $w_j\neq w_i$ do not affect the blip eigenvalues around $Nw_i/k$. 

For example, if we choose $W_1 = (1,-2,-2,3,3,3)$ and $W_2 = (0,0,0,3,3,3)$, then numerically we can observe that the histograms (Figure \ref{fig:blip122333} and \ref{fig:blip000333}) of the eigenvalues of the $500\times 500$ $(6,W_1)$-checkerboard matrices and $(6, W_2)$-checkerboard matrices at the blip around $\frac{1}{\sqrt{N}}\frac{N\cdot 3}{6} = \frac{1}{\sqrt{500}}\frac{500\cdot 3}{6}  \approx 11.2$ after normalization have approximately the same shape. 

\begin{figure}[H]
\begin{center}
\includegraphics[width = 16cm, height = 7 cm]{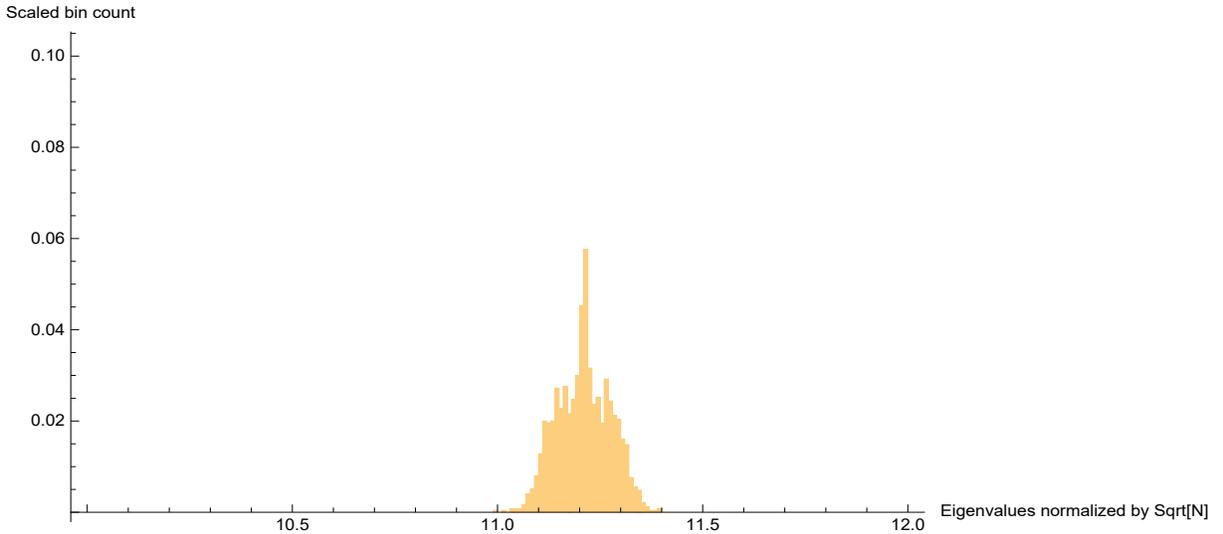}
\caption{\label{fig:blip122333} A histogram of the largest blip on a probability density scale for $500 \times 500$ $(6, W_1)$-checkerboard real matrices with $W_1 = (1,-2,-2,3,3,3)$ after 500 trials.}
\end{center}
\end{figure}

\begin{figure}[H]
\begin{center}
\includegraphics[width = 16cm, height = 7 cm]{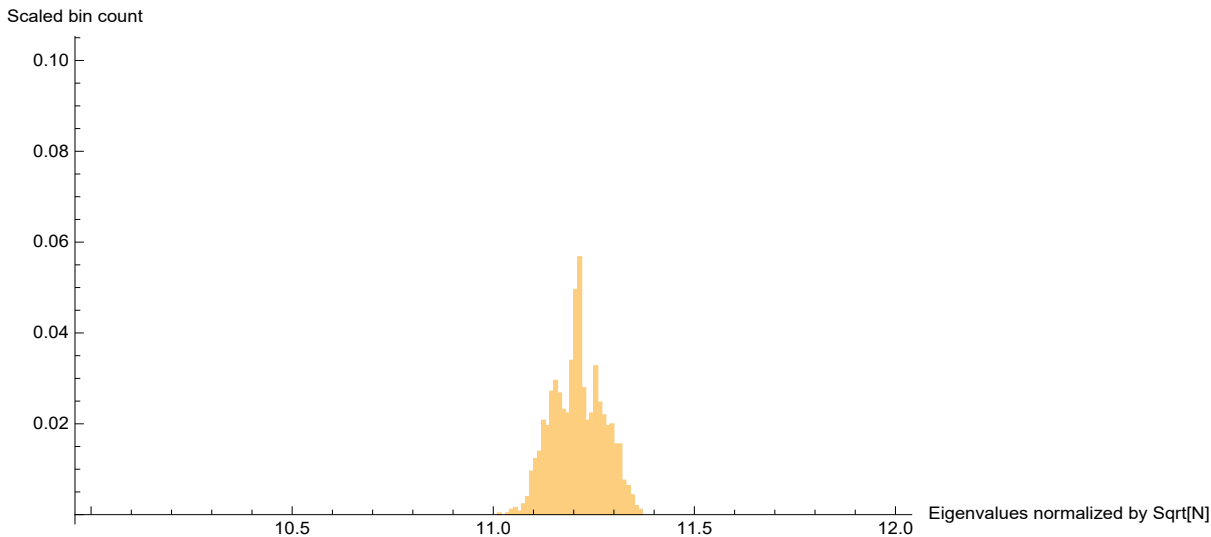}
\caption{\label{fig:blip000333} A histogram of the largest blip on a probability density scale for $500 \times 500$ $(6, W_2)$-checkerboard real matrices with $W_2 = (0,0,0,3,3,3)$ after 500 trials.}
\end{center}
\end{figure}

In particular, when there is only one eigenvalue in a blip, we obtain the following.

\begin{thm}
Fix $k\in\mathbb{N}$ and a $k$-tuple of real numbers $W = (w_1,\dots, w_k)$ where $w_i\neq 0$ and there is exactly one $w_i$ in $W$. Let $\{A_N\}_{N\in\mathbb{N}}$ be a sequence of $(k, W)$-checkerboard matrices. Then the associated empirical blip spectral measure $\mu_{A_N,i}$ around $Nw_i/k$ converges weakly to the Dirac delta distribution centered at $\frac{Nw_i}{k} + \frac{k-1}{w_i}$.
\end{thm}

Thus, when $k_i = 1$, we expect an eigenvalue of magnitude exactly $\frac{Nw_i}{k} + \frac{k-1}{w_i}$ as $N\to\infty$. In general, when $k_i > 1$, the empirical blip spectral measure of one matrix $A_N$ around $Nw_i/k$ no longer converges to the expected value, as the variances of the moments do not necessarily converge to zero as $N\to\infty$. Thus, we follow \cite{BCDHMSTPY} to modify the moment convergence theorem and average over the eigenvalues of multiple independent matrices.

\begin{defi}\label{def_ave_blip_fin}
Fix $k\in\mathbb{N}$, a $k$-tuple of real numbers $W = (w_1,\dots, w_k)$, and a function $g: \N \rightarrow \N$. The \textbf{averaged empirical blip spectral measure} around $Nw_i/k$ associated to a $g(N)$-tuple of $N \times N$ $(k,W)$-checkerboard matrices $(A_N^{(1)},A_N^{(2)},\dots,A_N^{(g(N))})$ is
\begin{equation}
\mu_{i,g,A_N^{(1)},A_N^{(2)},\dots,A_N^{(g(N))}}\ :=\ \frac{1}{g(N)}\sum_{j=1}^{g(N)} \mu_{A_N^{(j)},i}.
\end{equation}
\end{defi}

\begin{thm}\label{thm_main_goe}
Fix $k\in\mathbb{N}$, a $k$-tuple of real numbers $W = (w_1,\dots, w_k)$. Let $g: \N \rightarrow \N$ be such that there exists a $\delta>0$ for which $g(N) \gg N^\delta$. Let $A^{(j)}=\{A^{(j)}_N\}_{N \in \N}$ be sequences of fixed $N \times N$ matrices, and let $\overline{A}=\{A^{(j)}\}_{j \in \N}$ be a sequence of such sequences.
Then, as $N\to\infty$, the averaged empirical blip spectral measures $\mu_{i, g,A_N^{(1)},A_N^{(2)},\dots,A_N^{(g(N))}}$ around $Nw_i/k$ of the $(k,W)$-checkerboard ensemble over $\R$ converge weakly almost-surely to the measure with moments equal to the expected moments of the standard empirical spectral measure of the $k_i\times k_i$ hollow Gaussian Orthogonal Ensemble, where $k_i$ is the number of $w_i$ in $W$.
\end{thm}


In conclusion, we can construct an expanding family to have blips of any desired finite size at any sequence of positions after normalization. For example, we give an explicit construction of an ensemble whose limiting spectral measure has a semi-circle bulk and blips at all the Fibonacci numbers in Appendix \ref{sec:fibonacci}. Moreover, we extend \cite{BCDHMSTPY} by showing that the averaged empirical blip spectral measure around $\frac{Nw_i}{k}$ converges to a $k_i\times k_i$ hollow Gaussian with its mean $\frac{k-1}{w_i}$ independent of the choice of all the constants $w_j\neq w_i$. This measn that the distribution of different blips don't interfere with each other. When the blip has size $k_i = 1$, we get weak convergence of empirical blip spectral measure around $\frac{Nw_i}{k}$ to a Dirac Delta distribtion.


The paper is organized as follows. In \S\ref{sec:bulk} we prove our claims concerning the eigenvalues in the bulk, then turn to the blip spectral measure in  \S\ref{sec:blip}. We then prove results on the convergence to  the limiting spectral measure in \S\ref{sec:convergence}.


\section{The Bulk Spectral Measure and the split behavior}\label{sec:bulk}

\subsection{Bulk Measure}\label{bulk}
In this section we establish that the limiting bulk measure for the generalized $k$-checkerboard matrices follows a semi-circle. We denote by $\mu^{(m)}$ the $m$\textsuperscript{th} moment of the measure $\mu$.

The commonly used method of moments cannot be directly applied to our ensemble because as proved in \cite{BCDHMSTPY} the limiting expected moments of the empirical spectral measure do not exist. As remarked in the introduction, we overcome this difficulty by connecting the limiting spectral measure of the generalized checkerboard ensemble with that of $N \times N (k,0)$ checkerboard ensembles through Proposition \ref{stabilityESDrank}. \cite{BCDHMSTPY} then use the method of moments to establish the result for $(k,0)$-checkerboard matrices, using the eigenvalue trace lemma and combinatorics to establish convergence of the expected moments. The remaining arguments establishing almost sure weak convergence are standard (see for example Appendix A of \cite{BCDHMSTPY}); we state the result below.

\begin{lemma}\label{lem:avgmomentsbulk1}
The expected moments of the bulk empirical spectral measure taken over $A_N$ in the $N \times N$ $(k,0)$-checkerboard ensemble converge to the moments of the Wigner semicircle distribution $\sigma_R$ as defined in \eqref{wignersemi} with radius $R=2\sqrt{1-1/k}$ and
\begin{equation}
\E\left[\nu^{(l)}_{A_N}\right] \rightarrow \sigma_R^{(l)}
\end{equation}
as $N \rightarrow \infty$.
\end{lemma}

Thus, by Lemmas \ref{lem:avgmomentsbulk1} and \ref{stabilityESDrank}, we obtain the limiting distribution for the bulk of the general checkerboard ensemble.
\begin{lemma}\label{lem:avgmomentsbulk2}
The expected moments of the bulk empirical spectral measure taken over $A_N$ in the $N \times N$ general $(k,W)$-checkerboard ensemble converge to the moments of the Wigner semicircle distribution $\sigma_R$ with radius $R=2\sqrt{1-1/k}$ and
\begin{equation}
\E\left[\nu^{(\ell)}_{A_N}\right] \rightarrow \sigma_R^{(\ell)}
\end{equation}
as $N \rightarrow \infty$.
\end{lemma}
\subsection{Split Behavior}
In this section we demonstrate that general checkerboard matrices with $s$ different non-zero $w$'s almost surely have $s+1$ regimes of eigenvalues. One is $O(N^{1/2+\epsilon})$ (the bulk) and the others are of order $N$ (the blip). Similar to Appendix B of \cite{BCDHMSTPY}, we rely on matrix perturbation theory. In particular, we view a $(k,w)$-checkerboard matrix as the sum of a $(k,0)$-checkerboard matrix and a fixed matrix $Z$ where $Z_{ij}=w_u\chi\{{i \equiv j \equiv u \bmod k}\}$. In that sense, we view the $(k,w)$-checkerboard matrix as a perturbation of the matrix $Z$. Then, as the spectral radius of the $(k,0)$-checkerboard matrix is $O(N^{1/2+\epsilon})$, we obtain by standard results in the theory of matrix perturbations that the spectrum of the $(k,w)$-checkerboard matrix is the same as that of matrix $Z$ up to an order $N^{1/2+\epsilon}$ perturbation.

We begin with the following observation on the spectrum of the matrix $Z$.

\begin{lemma}\label{lem_spectrum_Z}
Suppose that $W$ has $x$ non-zero entries and
suppose that $w'_i$ appears $k_i$ times in $W=(w_1, \dots,w_k)$, then the matrix $Z$ has exactly $k_i$ eigenvalues at $Nw'_i/k$ and has $N-x$ eigenvalues at zero.
\end{lemma}
\begin{proof}
Suppose $w_{i_1}= w_{i_2}=\dots=w_{i_t}=w_{i'}$, then
for $1\le j \le t$ the vectors $\sum_{i=0}^{(N-1)/k} e_{ki+i_j}$ are eigenvectors with eigenvalues $Nw_{i'}/k$. Furthermore, for $1\le i\le N$ and $1\le j<k$  the vector $e_{ki+j}-e_{ki+j+1}$ are eigenvectors with eigenvalues equal to $0$.
\end{proof}

Weyl's inequality gives the following.

\begin{lemma}\label{lem_weyl}
(Weyl's inequality) \cite{HJ} Let $H,P$ be $N\times N$ Hermitian matrices, and let the eigenvalues of $H$, $P$, and $H+P$ be arranged in increasing order. Then for every pair of integers such that $1\le j,k\le N$ and $j+k\ge N+1$ we have
\begin{equation}
\lambda_{j+k-N}(H+P)\ \le\ \lambda_j(H)+\lambda_k(P),
\end{equation}\label{eqn_wyle1}
and for every pair of integers $j,k$ such that $1\le j,k\le N$ and $j+k\le N+1$ we have
\begin{equation}\label{eqn_wyle2}
\lambda_j(H)+\lambda_k(P)\ \le\ \lambda_{j+k-1}(H+P).
\end{equation}
\end{lemma}

Let $\|P\|_{\text{op}}$ denote $\max_i \left|\lambda_i(P)\right|$. By using the fact that $\left|\lambda_k (P)\right|\le \|P\|_{\text{op}}$ and taking $k=N$ in \eqref{eqn_wyle1}, we obtain that $\lambda_j(H+P)\le \lambda_j(H)+\|P\|_{\text{op}}$. Taking $k=1$ in \eqref{eqn_wyle2} gives the inequality on the other side, hence $\left| \lambda_j(H+P)-\lambda_j(H)\right|\le \|P\|_{\text{op}}$.

The above lemma implies that if the spectral radius of $P$ is $O(f)$ then the size of the perturbations are $O(f)$ as well. Hence it suffices to demonstrate that almost surely the spectral radius of a sequence of $(k,0)$-checkerboard matrices is $O(N^{1/2+\epsilon})$.

Let $A_N$ be a $(k,0)$-checkerboard matrix. By Remark A.3 in \cite{BCDHMSTPY} we have that $\var(\tr(A^{2m}_N))=O(N^{2m})$ and by the proof of Lemma \ref{lem:avgmomentsbulk1} we get $\E\left[\tr(A^{2m}_N)\right]=O( N^{m+1})$.

Since Lemma B.2 in \cite{BCDHMSTPY} holds for all $m\in \Z^+$, we have that almost surely $\|A_N\|_{\text{op}}$ is $O(N^{{1/2}+\epsilon})$. Together with Lemma \ref{lem_spectrum_Z} and Lemma \ref{lem_weyl}, we obtain the following.

\bulkregimes*

\section{The Blip Spectral Measure}\label{sec:blip}
In this section, we study the distribution of the eigenvalues at the blips. First, we define a weight function to enable us to focus on just one blip at a time. Then, we reduce the general cases to the case where all $w_j \not=w_i$ are zero. Finally, we show that the distribution in the special case is hollow $k_1 \times k_1$ gaussian following an argument similar to the one in \cite{BCDHMSTPY}.

Without loss of generality, we focus on the blip around $Nw_1/k \neq 0$ and use the polynomial weight function \begin{equation}f_1^{2n}(x)\ =\ \left(\frac{x(2-x)\prod_{w_{j}\neq w_1}(x-\frac{w_j}{w_1})(2-x-\frac{w_j}{w_1})}{\prod_{w_{j}\neq w_1}(1-\frac{w_j}{w_1})^2}\right)^{2n}.\end{equation} As discussed in Remark \ref{rmk:long_blip_justification}, the particular choice of weight functions does not change the result, provided that they are essentially $1$ close to $1$ and vanish to sufficiently high order at $0$ and all $w_j/w_1$ to remove the contribution from the eigenvalues within the bulk and the other blips.

\begin{defi} \label{def:blipmeas}
The \textbf{empirical blip spectral measure} associated to an $N \times N$ $k$-checkerboard matrix $A_N$ around $Nw_1/k$ is
\begin{equation}\label{eqn:blipmeas}
\mu_{A_N,1}(x) \; :=\;\frac{1}{k_1} \sum_{\lambda \text{ \rm an eigenvalue of } A}  f_1^{2n}\left(\frac{k\lambda}{w_1 N}\right)\delta\left(x-\left(\lambda-\frac{w_1N}{k}\right)\right),
\end{equation} where $k_1$ is the number of $w_1$'s in $(w_1,\dots, w_k)$, and $n(N)$ is a function satisfying $\lim\limits_{N\to\infty} n(N) = \infty$ and $n(N) = O(\log \log N)$.
\end{defi}

Because the fluctuation of the location of the eigenvalues in each regime is of order $\sqrt{N}$, the modified spectral measure of Definition \ref{def:blipmeas} weights eigenvalues within this blip by almost exactly 1 and those in the bulk and the other blips by almost exactly zero.

For fixed $N$, the polynomial $f_1^{2n}$ can be written as $
    f_1^{2n}(x)  = \displaystyle  \sum_{\alpha = 2n}^{4nl} c_{\alpha} x^{\alpha}
$, where $l$ is the number of distinct constants in $(w_1,\dots, w_k)$ and all $c_{\alpha}\in\mathbb{R}$.

We apply the method of moments to the modified spectral measure \eqref{eqn:blipmeas}. By the eigenvalue trace formula and linearity of expectation, the expected $m$-th moment of the empirical blip spectral measure is \begin{align}\label{eqn:blipmoment}
\mathbb{E}\left[\mu_{A_N,1}^{(m)}\right] \;=\; &\mathbb{E}\left[\frac{1}{k_1}\sum_{\lambda} \sum_{\alpha=2n}^{4nl}c_{\alpha}\left(\frac{k\lambda}{w_1 N}\right)^{\alpha}\left(\lambda-\frac{w_1N}{k}\right)^{m}\right] \nonumber\\
=\;&\mathbb{E}\left[\frac{1}{k_1}\sum_{\alpha = 2n} ^{4nl}c_{\alpha} \left(\frac{k}{w_1N}\right)^{\alpha}\left(\sum_{i = 0}^{m}\binom{m}{i} \left(-\frac{w_1N}{k}\right)^{m-i}\text{Tr}(A_N^{\alpha+i})\right)\right] \nonumber\\
=\; &\frac{1}{k_1}\sum_{\alpha = 2n} ^{4nl}c_{\alpha} \left(\frac{k}{w_1N}\right)^{\alpha}\left(\sum_{i = 0}^{m}\binom{m}{i} \left(-\frac{w_1N}{k}\right)^{m-i}\mathbb{E}\left[\text{Tr}(A_N^{\alpha+i})\right]\right).
\end{align}

Recall that
\begin{equation} \label{eqn:trace}
\mathbb{E}\left[\Tr( A_N^{\alpha+i})\right] \; =\; \sum_{1 \leq j_1, \ldots, j_{\alpha+i} \leq N} \mathbb{E}\left[m_{j_1 j_2}m_{j_2 j_3} \cdots m_{j_{\alpha+i} j_1}\right].
\end{equation}
The calculation of the moment has been transformed into a combinatorial problem of counting different types of products of entries. We follow the vocabulary from \cite{BCDHMSTPY} to describe the combinatorics problem. 

\begin{defi}\label{def:block}
A \textbf{block} is a set of adjacent $a$'s surrounded by $w$'s in a cyclic product, where the last entry of a cyclic product is considered to be adjacent to the first. We refer to a block of length $\ell$ as an $\ell$-block or sometimes a block of size $\ell$.
\end{defi}

\begin{defi}\label{def:configuration}
A \textbf{configuration} is the set of all cyclic products for which it is specified (a) how many blocks there are, and of what lengths, and (b) in what order these blocks appear (up to cyclic permutation);
However, it is not specified how many $w$'s there are between each block.
\end{defi}

\begin{defi}\label{def:congruenceconfiguration}
A \textbf{congruence configuration} is a configuration together with a choice of the congruence class modulo $k$ of every index.
\end{defi}

\begin{defi}\label{def:matching}
Given a configuration, a \textbf{matching} is an equivalence relation $\sim$ on the $a$'s in the cyclic product which constrains the ways of indexing (see Definition \ref{def of an Indexing}) the $a$'s as follows: an indexing of $a$'s conforms to a matching $\sim$ if, for any two $a$'s $a_{i_{\ell},i_{\ell+1}}$ and $a_{i_{t},i_{t+1}}$, we have $\{i_\ell,i_{\ell+1}\}=\{i_t,i_{t+1}\}$ if and only if $a_{i_{\ell}i_{\ell+1}} \sim a_{i_{t},i_{t+1}}$. We further constrain that each $a$ is matched with at least one other by any matching $\sim$.
\end{defi}

\begin{defi}\label{def of an Indexing}
Given a configuration, matching, and length of the cyclic product, then an \textbf{indexing} is a choice of
\begin{enumerate}
\item the (positive) number of $w$'s between each pair of adjacent blocks (in the cyclic sense), and
\item the integer indices of each $a$ and $w$ in the cyclic product.
\end{enumerate}
\end{defi}

\begin{exa}\label{ex:congruence}
Consider the configuration
\begin{equation}
\cdots a_{i_1i_2}w_{i_2i_3}w_{i_3i_4}a_{i_4i_5}a_{i_5i_6} \cdots.
\end{equation}
Then we have
\begin{equation}
i_2 \equiv i_3 \equiv i_4  \pmod{k}.
\end{equation}
\end{exa}

We see that the congruence classes of the indices of the $a$'s determine which congruence classes of the indices of the $w$'s belong to, and thus which $w_j$'s appear between the blocks.

\subsection{Reducing to the case where all $w_j\neq w_1$ are zero}

We will show that if there is some $w_j\neq w_1$ in a fixed congruence configuration
, then it does not contribute to the expected moment \eqref{eqn:blipmoment} in the limit.

We begin by analyzing the form of the summands in the total contribution of a congruence configuration in Lemma \ref{lem:sumform}. The following lemma helps us to derive this form, and its proof is provided in Appendix \ref{sec:appendixproof1}.

\begin{lemma} \label{lem:poly}
	Fix $s\in\mathbb{N}$ with $s\geq 2$ and some polynomial $p(x_1,\dots, x_s)\in\mathbb{R}[x_1,\dots, x_s]$ of degree $q$. For $\eta\in\mathbb{N}$ with $\eta\geq \sum_{i=1}^{s} y_i$ and distinct $w_1,\dots, w_s$, we have \begin{align}
	\sum_{\substack{x_1+\dots + x_s = \eta\\ x_i\geq y_i}}p(x_1,\dots, x_s) w_1^{x_1} \cdots w_{s}^{x_s} \; = \; \frac{\sum_{l = 1}^{s} w_{l}^{\eta + 2 -\sum_{i=1}^{s}y_i} f_{l,\eta}(w_1,\dots, w_s)}{(\prod_{1\leq i < j \leq s}(w_i - w_j))^{2^{q}}},
	\end{align} where each $f_{l,\eta}(w_1,\dots, w_s)\in\mathbb{R}[\eta][w_1,\dots, w_s]$ is a homogeneous polynomial in $w_1,\dots, w_s$ of degree $ 2^{q}\binom{s}{2} + (\sum_{i=1}^{s}y_i)-2$. Furthermore, the coefficients in the polynomial $f_{l,\eta}(w_1,\dots, w_s)$ are polynomials in $\eta$ of degree $\leq q$.
\end{lemma}

\begin{lemma} \label{lem:sumform}
    Fix a congruence configuration and a matching. 
    The contribution to $\mathbb{E}\left[ \Tr (A_N^{\alpha+i}) \right]$ is a sum of terms of the form \begin{equation} \label{eqn:sumform}
    p(\alpha+i) w_j^{\alpha + i-\gamma}\left(\frac{N}{k}\right)^{\alpha + i - t}\end{equation} where $p$ is a polynomial of degree $\leq \beta-s+1$, $\beta$ is the number of blocks determined by the configuration, $s$ is the number of distinct constants $w$'s determined by the chosen congruence classes, and $t$ is the lost degrees of freedom determined by the matching.
\end{lemma}
\begin{proof}
    Suppose that the distinct constants $w_{j_1}$, $w_{j_2}$, $\dots$, $w_{j_s}$ appear in the configuration, where each $w_{j_q}$ appears $x_q$ times and the $x_q$ $w_{j_q}$'s are separated by the blocks into $y_q$ parts. There are $\binom{x_q-1}{y_q-1}$ ways to put $x_q$ $w_{j_q}$ into $y_q$ gaps.
    Note that $\mathbb{E}\left[m_{j_1 j_2}m_{j_2 j_3} \cdots m_{j_{\alpha+i} j_1}\right] = w_{j_1}^{x_1}x_{j_2}^{x_2} \cdots w_{j_s}^{x_s} \mathcal{A}$ where $\mathcal{A}$ is some constant determined by the matching.

    Denote the number of $a$'s in this configuration by $r$, then $x_1+x_2+\dots+x_s=\alpha+i-r$ and there are \begin{equation} \label{eqn:combinatorics}
    \sum_{\substack{x_i \geq y_i \\ x_1+\dots+x_s=\alpha+i-r}}\prod_{q=1}^s\binom{x_q-1}{y_q-1}w_{j_q}^{x_q}\end{equation} ways to place the constants $w_{j_1}$, $w_{j_2}$, $\dots$, $w_{j_s}$.
    For fixed $y_1,\dots, y_s$, we can write \eqref{eqn:combinatorics} as \begin{equation} \label{eqn:combinatorics_poly}
    \sum_{\substack{x_i \geq y_i \\ x_1+\dots+x_s=\alpha+i-r}}\widetilde{g}_{y_1,\dots, y_s} (x_1,\dots, x_s)w_{j_1}^{x_1}w_{j_2}^{x_2}\cdots w_{j_s}^{x_s}\end{equation} where $\widetilde{g}_{y_1,\dots, y_s}(x_1,\dots, x_s) \in \mathbb{R}[x_1,\dots, x_s]$ is a polynomial in $x_1,\dots, x_s$ of degree $\sum_{q=1}^{s}(y_q-1) = (y_1+\cdots+y_s) - s = \beta-s$.

    By Lemma \ref{lem:poly}, we can write \eqref{eqn:combinatorics_poly} as a sum of the terms of the form \begin{equation}w_j^{\alpha+i-r+2-\sum_{i=1}^{s} y_i} \widetilde{p}(\alpha+i-r) \;=\; \widetilde{p}(\alpha+i-r)
    w_j^{\alpha+i-r+2-\beta}
    \end{equation} where $\widetilde{p}(x) \in\mathbb{R}[x]$ is a polynomial of degree $\leq \beta-s$.

    Recall that $\beta,r$ are constants fixed by the configuration. Taking into account cyclic permutation, the contribution is a sum of the terms of the form \begin{equation*}
    p(\alpha+i)w_j^{\alpha+i-\gamma}\left(\frac{N}{k}\right)^{\alpha+i-t}
    \end{equation*} where $p(x) \in\mathbb{R}[x]$ is a polynomial of degree $\leq \beta-s+1$, $\gamma\in\mathbb{Z}$ and $\left(\frac{N}{k}\right)^{\alpha+i-t}$ is from choosing the indices from given equivalence classes modulo $k$.
\end{proof}

Observe that in \eqref{eqn:trace} there are $(\alpha+i)$ degrees of freedom in choosing $j_1,\dots, j_{\alpha+i}$. Whenever the lost degrees of freedom $t \geq m+1$, we have
\begin{align}\label{equation 3.5}
&\sum_{\alpha = 2n} ^{4nl}c_{\alpha} \left(\frac{k}{w_1N}\right)^{\alpha}\left(\sum_{i = 0}^{m}\binom{m}{i} \left(-\frac{w_1N}{k}\right)^{m-i}N^{\alpha+i-t}\right) \;\nonumber\\ = \;& N^{m-t} \left(\sum_{\alpha = 2n} ^{4nl}c_{\alpha} \left(\frac{k}{w_1}\right)^{\alpha}\right)\left(\sum_{i = 0}^{m}\binom{m}{i} \left(-\frac{w_1}{k}\right)^{m-i}\right) \nonumber\\
\ll \;& N^{m-t} \left(1+\frac{|w_1|}{k}\right)^{m} \left\lvert f_1^{2n}\left(\frac{k}{w_1}\right)\right\rvert \nonumber\\
\ll \;& N^{-1}  \left\lvert f_1\left(\frac{k}{w_1}\right)\right\rvert^{2n},
\end{align} then since we have required $n(N) = O(\log \log N)$, we only need to consider the contribution from $\mathbb{E}\left[ A_N^{\alpha+i}\right]$ that loses at most $m$ degrees of freedom.

\begin{rem}
Even though each term contributes $O(1/N)$, the contribution adds up to $c^{n(N)}/N$ for some $c\in\mathbb{R}$. Thus, in order to remove the contributions from configurations with more than $m$ blocks in this way, we have to require $n = o(\log N)$, so we correct the assumed growth rate $ n(N) \gg N^{\epsilon} $ in \cite{BCDHMSTPY}. 
\end{rem}

We cite the following lemma from \cite{BCDHMSTPY}, which relates the number of blocks to the lost degree of freedom.
\begin{lem}\label{lem_contributions}(\cite{BCDHMSTPY})
Fix the number of blocks $\beta$, and consider all classes with $\beta$ blocks. Then the classes among these with the highest number of degrees of freedom are exactly those which contain only $1$- or $2$-blocks, $1$-blocks are matched with exactly one other $1$-block, and both $a$'s in any $2$-block are matched with their adjacent entry and no others.
\end{lem}

\begin{rem}\label{Remark of Lemma 3.11}
In \cite{BCDHMSTPY}, they prove Lemma \ref{lem_contributions} by showing that the average number of degrees of freedom lost per block is at least $1$, and that the average number of degrees of freedom lost per block is $1$ if and only if we have the configurations and matchings specified in Lemma \ref{lem_contributions}. 
\end{rem}

By Lemma \ref{lem_contributions}, we can restrict ourselves to the configurations that have no more than $m$ blocks.

The following lemma allows us to cancel the contributions from the congruence configurations that contain some constants $w_j\neq w_1$ and reduce the general case to the special one where all the constant $w_j\neq w_1$ are zero.

\begin{lemma}\label{lem_zero}
    Suppose the polynomial $f(x)  :=  \sum_{\alpha} c_{\alpha} x^{\alpha} \in \mathbb{R}[x]$ has a zero of order $n > 0$ at $x_0$. Then \begin{align}
    \sum_{\alpha}c_{\alpha}x_0^{\alpha} p(\alpha) \; = \; 0
\end{align} for any polynomial $p$ of degree $d < n$.
\end{lemma}

The lemma is proved in Appendix \ref{sec:appendixproof2}.

We are now ready to show that the contributions from the congruence configurations that contain some constants $w_j\neq w_1$ cancel. 

Given any polynomial $p(x) \in\mathbb{R}[x]$ and $\gamma, t\in\mathbb{Z}$, notice the following
\begin{enumerate}
\item If $w_j\neq w_1$, $w_j\neq 0$, and $p$ has degree less than $2n$, then \begin{align} \label{eqn:eqn1}
&\sum_{\alpha = 2n} ^{4nl}c_{\alpha} \left(\frac{k}{w_1N}\right)^{\alpha}\left(\sum_{i = 0}^{m}\binom{m}{i} \left(-\frac{w_1N}{k}\right)^{m-i}p(\alpha+i) w_j^{\alpha + i-\gamma}\left(\frac{N}{k}\right)^{\alpha + i - t}\right) \nonumber\\
=\; & \frac{k^t}{w_j^{\gamma} N^t}\sum_{\alpha = 2n}^{4nl} c_{\alpha}  \left(\frac{w_j}{w_1}\right)^{\alpha}
\left(\sum_{i = 0}^{m}\binom{m}{i} \left(-\frac{w_1N}{k}\right)^{m-i}p(\alpha+i)\left(\frac{w_jN}{k}\right)^{i}\right) \nonumber\\
=\; & \frac{k^t}{w_j^{\gamma} N^t}\sum_{i=0}^{m} \binom{m}{i}\left(-\frac{w_1N}{k}\right)^{m-i}\left(\frac{w_jN}{k}\right)^{i}\sum_{\alpha = 2n}^{4nl} c_{\alpha} \left(\frac{w_j}{w_1}\right)^{\alpha}
p(\alpha+i) \nonumber\\
=\; & 0, \end{align} where we get $\sum_{\alpha = 2n}^{4nl} c_{\alpha} (\frac{w_j}{w_1})^{\alpha}
p(\alpha+i) = 0$ from Lemma \ref{lem_zero} using the fact that $f_1^{2n}(x)  =   \sum_{\alpha = 2n}^{4nl} c_{\alpha} x^{\alpha} \in \mathbb{R}[x]$ has a zero of order $2n$ at $w_j/w_1$.
\item If $w_j = w_1$, and $p$ has degree less than $m$, then
\begin{align} \label{eqn:eqn2}
&\sum_{\alpha = 2n} ^{4nl}c_{\alpha} \left(\frac{k}{w_1N}\right)^{\alpha}\left(\sum_{i = 0}^{m}\binom{m}{i} \left(-\frac{w_1N}{k}\right)^{m-i}p(\alpha+i) w_1^{\alpha + i-\gamma}\left(\frac{N}{k}\right)^{\alpha + i - t}\right) \nonumber\\
=\; & \frac{k^{t-m}}{w_1^{\gamma-m} N^{t-m}}\sum_{\alpha = 2n}^{4nl} c_{\alpha}
\left(\sum_{i = 0}^{m}\binom{m}{i} (-1)^{m-i}p(\alpha+i)\right) \nonumber\\
=\; &0,
\end{align}
where we get $\sum_{i = 0}^{m}\binom{m}{i} (-1)^{m-i}p(\alpha+i) = 0$ from Lemma \ref{lem_zero} using the fact that $(x-1)^m$ has a zero of order $m$ at $1$.
\end{enumerate}

By Lemma \ref{lem:sumform}, we know given a configuration with $\beta$ blocks, the polynomial $p$ in the contribution \eqref{eqn:sumform} has degree $\leq \beta -s +1 $ where $s$ the number of distinct $w$'s in this configuration determined by the chosen congruence classes. In particular, given $\beta\leq m$, the polynomial $p$ has degree $\leq m -1+1 = m$, and whenever both $w_1$ and $w_j\neq w_1$ appear in the configuration, the polynomial $p$ has degree $\leq m- 2 + 1 = m-1$.

From \eqref{eqn:eqn1} and \eqref{eqn:eqn2}, we conclude that the configurations with some $w_j\neq w_1$ do not contribute to the moment.
We may therefore assume that all all $w_j\neq w_1$ are zero. 

\subsection{The special case where all $w_j\neq w_1$ are zero}

We have reduced to the special case where $k_1$ of the $w_j$'s are $w_1$ and the rest $k-k_1$ are $0$.

Following the arguments in \S 3 of \cite{BCDHMSTPY}, we can show that the contributions to the $m$-th moment from all configurations with fewer than $m$ blocks cancel, and the contributions from all configurations with matchings that lose more than $m$ degrees of freedom become insignificant as $N\to\infty$. In particular, by Lemma \ref{lem_contributions}, we are only left with the configurations with $m$ blocks.


\begin{proposition}\label{proposition contribution}
Fix the number of blocks $\beta$, the total contribution of configurations with $m_1$ 1-blocks to $\mathbb{E}\left[\Tr( A_N^{\alpha+i})\right]$ is
\begin{align}
&w_1^{\alpha+i-m_1-2(\beta-m_1)}
\left(\frac{(\alpha+i)^{\beta}}{\beta!}+\tilde{p}(\alpha+i)\right)\binom{\beta}{m_1}
(k-1)^{\beta-m_1}
\mathbb{E}_{k_1}\left[\text{Tr}(B^{m_1})\right]
\left(\frac{N}{k}\right)^{\alpha+i-\beta} \nonumber\\
&+O_{\beta}\left((\alpha+i)^{\beta}
\left(\frac{N}{k}\right)^{\alpha+i-\beta-1}
\right)
\end{align}
where $\tilde{p}$ is a polynomial of degree $\leq \beta-1$.
\end{proposition}
The proof follows closely from that of Proposition 3.15 in \cite{BCDHMSTPY}, and is given in Appendix \ref{app:blip}.

\begin{proposition} \label{prop:expmoment}
    The expected $m$-th moment in the limit is \begin{equation}
    \lim\limits_{N\to\infty}\mathbb{E}\left[\mu_{A_N,1}^{(m)}\right] \; = \;\frac{1}{k_1}\sum_{m_1=0}^m\binom{m}{m_1}\left(\frac{k-1}{w_1}\right)^{m-m_1}\mathbb{E}_{k_1}\left[\text{Tr}(B^{m_1})\right].
\end{equation}
\end{proposition}
The proof follows closely from that of Theorem 3.18 in \cite{BCDHMSTPY}, and is given in Appendix \ref{app:blip2}. 

Note that the expected first moment in the limit is
\begin{equation}
    \lim\limits_{N\to\infty}\mathbb{E}\left[\mu_{A_N,1}^{(1)}\right] \; = \; \frac{k-1}{w_1}. \end{equation}

Following the same calculation as in Theorem 3.18 of \cite{BCDHMSTPY}, we obtain the centered $m$-th moment 
\begin{align}\label{moment}
    \mu_{c,1}^m:= \;
    \lim_{N \to \infty}\mathbb{E}\left[\int(x-\mu_{A_N,1}^{(1)})^m d\mu_{AN,1}\right] \;
    =\;\frac{1}{k_1}
    \mathbb{E}_{k_1}\left[\text{Tr}(B^{m})\right].
\end{align}

\section{Weak and Absolute Convergence of the Blip Spectra Measure} \label{sec:convergence}

In this section we establish the convergence result for the blip spectra measure. We first use the standard technique to show weak convergence for a blip of size one, and then we follow \cite{BCDHMSTPY} to established a modified spectral measure and prove its convergence.

\subsection{Weak convergence for blip of size $1$}

\begin{defi}(Weak Convergence).
A family of probability distribution $\mu_n$ weakly converges to $\mu$ if and only if for any bounded, continuous $f$ we have
\[\lim_{n \to \infty}\int_{-\infty}^{\infty} f(x)\mu_n(dx)=\int_{-\infty}^{\infty} f(x)\mu(dx).\]

\end{defi}
Since $\mu_{A_N, 1}^{(m)}$ is finite, to prove weak convergence we should prove the variance of expected $m^{\text{th}}$ tends to zero as $N$ goes to infinity. That is,
\[\lim_{N \to \infty}\mathbb{E}[(\mu^{(m)}_{A_N,1})^2]-\mathbb{E}[(\mu^{(m)}_{A_N,1})]^2=0.\]

By \eqref{eqn:blipmoment}, we have that
\begin{align}
    \mathbb{E}[(\mu^{(m)}_{A_N,1})^2] \: = \: & \frac{1}{k_1^2}\sum_{\alpha=2n}^{4nl}\sum_{\beta=2n}^{4nl}c_{\alpha}c_{\beta}\sum_{i=0}^m\sum_{j=0}^m{m \choose i}{m \choose j}(-1)^{i+j}(\frac{w_1N}{k})^{2m-(i+j)-(\alpha+\beta)}\nonumber\\ \: & \left(\sum_{\substack{C_{i+\alpha}\\C_{j+\beta}}}(\mathbb{E}[C_{i+\alpha}C_{j+\beta}]\right)\label{E^2 1}, \\
    \mathbb{E}\left[\mu_{A_N,1}^{(m)}\right]^2 \; = \; & \frac{1}{k_1^2} \sum_{\alpha = 2n}^{4n}\sum_{\beta = 2n}^{4n} c_{\alpha} c_{\beta} \sum_{i = 0}^{m}\sum_{j = 0}^{m} \binom{m}{i}\binom{m}{j}(-1)^{i+j}\left(\frac{w_1N}{k}\right)^{2m-(i+j)-(\alpha+\beta)}\nonumber\\ \, & \left(\sum_{\substack{C_{i+\alpha}, \\ C_{j+\beta}}}\mathbb{E}\left[C_{i+\alpha}\right]\mathbb{E}\left[C_{j+\beta}\right]\right)\label{E^2 2},
\end{align}
where $C_{t}$ denotes the cycle $m_{i_1i_2}m_{i_2i_3}\dots m_{i_ti_1}$
Notice that the difference cancels unless there exists $a_{t_1t_2}$ such that $a_{t_1t_2} \in C_{i+\alpha}$ and $C_{j+\beta}$. Therefore we only need to count the pair of cycles where $C_{i+\alpha}$ and $C_{j+\beta}$ has at least one common $a$. We call such pair of cycles the crossover terms.

\begin{lemma}
The contributions of crossover terms to $\mathbb{E}[(\mu^{(m)}_{A_N,1})]^2$ is 0 as $N \to \infty$.
\end{lemma}
\begin{proof}
$\mathbb{E}[(\mu^{(m)}_{A_N,1})]^2$ is the product of \[\frac{1}{k_1}\sum_{\alpha=2n}^{4nl}c_{\alpha}\sum_{i=0}^m {m \choose i}(-1)^{m-i}(\frac{w_1N}{k})^{m-i-\alpha}\sum_{C_{i+\alpha}}\mathbb{E}[C_{i+\alpha}]\]
and
\[\frac{1}{k_1}\sum_{\beta=2n}^{4nl}c_{\beta}\sum_{i=0}^m {m \choose j}(-1)^{m-j}(\frac{w_1N}{k})^{m-j-\beta}\sum_{C_{j+\beta}}\mathbb{E}[C_{j+\beta}].\]

Suppose we fix a pair of congruence configurations of $C_{i+\alpha}$ and $C_{i+\alpha}$ such that there is a common $a_{t_1t_2}$ in the two cycles, and $C_{i+\alpha}$ has $b_1$ blocks while $C_{j+\beta}$ has $b_2$ blocks. If either of $b_1, b_2$ is less than $m$, then by $\ref{eqn:eqn1}$ and $\ref{eqn:eqn2}$ their product makes $0$ contribution.
So we know the that configuration contributes only when $b_1+ b_2 \geq 2m$. By Lemma $\ref{lem_contributions}$, each block loses at least $1$ degree of freedom. However, due to the common $a_{t_1t_2}$ of $C_{i+\alpha}$ and $C_{j+\beta}$, there is a block that loses at least $2$ degrees of freedom, so in total at least $b_1+b_2+1 \geq 2m+1$ degree of freedom is lost. Thus by $\eqref{equation 3.5}$, the crossover terms contribute to 0 when $N \to \infty$.
\end{proof}
Now it is sufficient to look at the contribution from crossovers to
$\mathbb{E}[(\mu^{(m)}_{A_N,1})^2]$. For general $k_1$, the contributions of the crossovers doesn't necessarily go to $0$ as $N \to \infty$. We want to show that for $k_1=1$, the contribution from the crossovers does go to $0$. In order to show this, we first reduce the general $W$ to the simplest case where all $w_j\neq w_1$ are zero. 

\begin{lemma}\label{E^2 nothing}
The contribution from the congruence configurations that contain $w_j \neq w_1$ to  $\mathbb{E}[(\mu^{(m)}_{A_N,1})^2]$ is $0$.
\end{lemma}

\begin{proof}
Fix a pair of congruence configuration. Say $w_{i_1}, w_{j_2},\dots, w_{j_s}$
appears in the cyclic product
$C_{i+\alpha}$ and $w_{j_q}$ appears $x_q$ times, separated by the blocks into $y_q$ parts.
$w_{j'_{1}},\dots,w_{j'_{s'}}$ appears in the cyclic product $C_{j+\beta}$, $w_{j_q'}$ appears
${x'}_{q'}$ times and are separated by the blocks into $y_{q'}$ parts.
The sum of $y$s should be the total number of blocks, so we have $y_1+\cdots +y_s+{y'}_1+\cdots +{y'}_{s'}=b_1+b_2$.
By Lemma $\ref{lem_contributions}$, the total lost degree of freedom is at least $b_1+b_2$. On the other hand, by \eqref{equation 3.5} we know that the total lost of degree of freedom should be at most $2m$. Therefore we have $b_1+b_2 \leq 2m$.

By Lemma $\ref{lem:sumform}$, with the congruence configuration fixed, the total number of ways to place $w_j$ and $w_{j'}$ is
\begin{equation}(\alpha+i)\sum_{\substack{x_i \geq y_i \\ x_1+\cdots +x_s=\alpha+i-r_1}} \prod_{q=1}^s {x_q-1 \choose y_q-1}w_{j_q}^{x_q}(\beta+j)\sum_{\substack{x'_{i'}\geq y'_{i'}\\ x'_{1}+\cdots +x'_{s'}=\beta+i-r_2 }}\prod_{q'=1}^{s'} {x'_{q'}-1 \choose y'_{q'}-1}w_{j_q'}^{x_q'}\end{equation}
where $r_1,r_2$ are the number of $a$ in each cycle. Since $y_i, y'_{i'}$ are fixed, the above expression can be written as
\begin{align}\label{above equation}
&    (\alpha+i)\sum_{\substack{x_i \geq y_i \\ x_1+\cdots +x_s=\alpha+i-r_1}}p_{y_1,\dots, y_s}(x_1,\dots, x_s)w_{j_1}^{x_1}\cdots w_{j_s}^{x_s}(\beta+j) \nonumber\\ & \ \ \ \ \ \ \ \ \ \ \ \ \cdot \ \sum_{\substack{x'_{i'}\geq y'_{i'}\\ x'_{1}+\cdots +x'_{s'}=\beta+i-r_2 }}p_{y'_1,\dots, y'_{s'}}(x'_1,\dots, x'_{s'})w_{j'_1}^{x'_1}\cdots w_{j'_{s'}}^{x'_{s'}}
\end{align}
where  $p_{y_1,\dots, y_s}$ and $p_{y'_1,\dots, y'_{s'}}$ are polynomials with variables $x_i, x'_{i'}$, and the sum of their degree is  $y_1+\cdots +y_s+y'_1+\cdots +y'_{s'}-s-s'=b_1+b_2-s-s' $. Then \eqref{above equation} is a sum of terms of the form $p_1(\alpha+j)w_j^{\alpha+j-\gamma_1}p_2(\beta+i)w_{j'}^{\beta+i-\gamma}$, where sum of degrees of $p_1$ and $p_2$ is $b_1+b_2+2-s-s'$,  $s$ or $s'$ should be at least $2$. Since $s, s' \geq 1, b_1+b_2 \leq 2m$, the sum of degree of $p_1$ and $p_2$ would be at most $2m-1$. Therefore at least one of $p_1,p_2$ will have degree $\leq m-1$. Without loss of generality say $p_1$ has degree $\leq m-1$. Then by $\eqref{eqn:eqn1}$ and $\eqref{eqn:eqn2}$
\begin{equation}\sum_{\alpha = 2n} ^{4nl}c_{\alpha} \left(\frac{k}{w_1N}\right)^{\alpha}\left(\sum_{i = 0}^{m}\binom{m}{i} \left(-\frac{w_1N}{k}\right)^{m-i}p_1(\alpha+i) w_j^{\alpha + i-\gamma}\left(\frac{N}{k}\right)^{\alpha + i - t}\right)\ = \ 0. \end{equation}
Therefore, the contribution of the terms $p_1(\alpha+j)w_j^{\alpha+j-\gamma_1}p_2(\beta+i)w_{j'}^{\beta+i-\gamma}$ to will be 0. Thus the contribution from congruence configurations containing $w_j \neq w_1$ to $\mathbb{E}[(\mu^{(m)}_{A_N,1})^2]$ is 0 as $N \to \infty$.

\end{proof}

Now we can restrict ourselves to the simplest case where $w_j \neq w_1$ are all $0$. We want to prove that when $k_1=1$, the contribution from the crossovers to $\mathbb{E}[(\mu^{(m)}_{A_N,1})^2]$ is $0$.
Assume $w_1\neq 0$ and $w_2 = \dots = w_k = 0$.

\begin{thm}
When $k_1=1$ and $w_j=0$ for all $w_j \neq w_1$, we have
\begin{align}
		\lim\limits_{N\to\infty}\text{Var}\left[\mu_{A_N,1}^{(m)}\right] \; = \; 0.
	\end{align}
\end{thm}

\begin{proof}
We are left to prove that the contributions from crossovers to $\mathbb{E}[(\mu^{(m)}_{A_N,1})^2]$ is $0$.

Fix the pair of congruence configuration at $C_{i+\alpha}$ and $C_{j+\beta}$. Suppose there are $b_1$ blocks in $C_{i+\alpha}$ and $b_2$ blocks in $C_{j+\beta}$.

If $b_1< m$ or $b_2 < m$, then \begin{equation*}
	\sum_{i=0}^{m}\sum_{j = 0}^{m}\binom{m}{i}\binom{m}{j}(-1)^{2m-i-j}i^{p'}j^{q'} \;= \;	\sum_{i=0}^{m}\binom{m}{i}(-1)^{m-i}i^{p'}	\sum_{j = 0}^{m}\binom{m}{j}(-1)^{m-i}j^{q'} \;= \;0
	\end{equation*} for all integers $0\leq p'\leq b_1$ and $0\leq q'\leq b_2$, so that the contributions from this configuration cancel out. So we only need to look at configurations with $b_1 \geq m$ and $b_2 \geq m$.
	
Now notice that in this $W$, $w_j = 0$ for all $j\not\equiv 1 \;(k)$. Thus, if there is some 1-block in $C_{i+\alpha}$ or $C_{j+\beta}$, then both $\mathbb{E}\left[C_{i+\alpha}C_{j+\beta}\right]$ and $ \mathbb{E}\left[C_{i+\alpha}\right]$ $\mathbb{E}\left[C_{j+\beta}\right]$ are $0$. Therefore, we can restrict ourselves to the configurations where all the blocks are $2$-blocks.

By Lemma $\ref{lem_contributions}$ and Equation $\eqref{equation 3.5}$, We only need to consider congruence configurations where $b_1+b_2 \leq 2m$. Combining with $b_1 \geq m, b_2 \geq m$, We require $b_1=b_2=m$, and both $a$'s in 2-blocks matched with their adjacent entry. But then crossover matchings between $C_{i+\alpha}$ and $C_{j+\beta}$ become impossible. Therefore, we conclude that
\begin{align}
		\lim\limits_{N\to\infty}\text{Var}\left[\mu_{A_N,1}^{(m)}\right] \; = \; 0.
	\end{align}
\end{proof}
\subsection{Absolute Convergence of modified blip spectral measure }

We have computed the expected $m$-th moment $\mathbb{E}\left[\mu_{A_N,1}^{(m)}\right]$ of the empirical blip measure around $Nw_1/k$. However, given one matrix $A_N$ from our ensemble, the $m$-th moment $\mu_{A_N,1}^{(m)}$ of its empirical blip measure do not necessarily converge to this system average $\mathbb{E}\left[\mu_{A_N,1}^{(m)}\right]$, as we will show in the following example that $\lim\limits_{N\to\infty}\text{Var}\left[\mu_{A_N,1}^{(2)}\right] =\lim\limits_{N\to\infty}\left( \mathbb{E}\left[\left(\mu_{A_N,1}^{(2)}\right)^2\right]-\mathbb{E}\left[\mu_{A_N,1}^{(2)}\right]^2\right) > 0$.

\begin{exa}{\label{example}}
	Consider the special case where all $w_j\neq w_1$ are zero and $k_1> 1$. We can obtain the expression of $\text{Var}\left[\mu_{A_N,1}^{(2)}\right]$ by plugging $m=2$ into Equations \eqref{E^2 1} and \eqref{E^2 2}. By Lemma $\ref{E^2 nothing}$ we only need to calculate $\mathbb{E}\left[\left(\mu_{A_N}^{(2)}\right)^2\right]$.

The configurations that neither cancel out nor contribute insignificantly in the limit have two 1-blocks in both $C_{i+\alpha}$ and $C_{j+\beta}$ with cross-over matching. Its contribution to $\sum_{\substack{C_{i+\alpha}, \\ C_{j+\beta}}}\mathbb{E}\left[C_{i+\alpha}C_{\beta+j}\right]$ is given by \begin{align}
	2\binom{\alpha+i}{2}\binom{\beta+j}{2}k_1(k_1-1) \left(\frac{Nw_1}{k}\right)^{\alpha+i+\beta+j-4},
	\end{align} where $2\binom{\alpha+i}{2}\binom{\beta+j}{2}$ is from matching and choosing the positions of the 1-blocks and $k_1(k_1-1)$ is from choosing the equivalence classes of the indices.
	Then \begin{align}
	 \lim\limits_{N\to\infty}\text{Var}\left[\mu_{A_N,1}^{(2)}\right]
	\; =
	\; &\frac{1}{k_1^2} \sum_{\alpha = 2n}^{4n}\sum_{\beta = 2n}^{4n} c_{\alpha} c_{\beta} \sum_{i = 0}^{2}\sum_{j = 0}^{2} \binom{2}{i}\binom{2}{j}(-1)^{4-i-j}2\binom{\alpha+i}{2}\binom{\beta+j}{2}k_1(k_1-1)\nonumber\\
	=\; &\frac{2(k_1-1)}{k_1}\sum_{\alpha = 2n}^{4n}\sum_{\beta = 2n}^{4n} c_{\alpha} c_{\beta} \sum_{i = 0}^{2}\binom{2}{i}(-1)^{2-i}\binom{\alpha+i}{2}\sum_{j = 0}^{2} \binom{2}{j}(-1)^{2-j}\binom{\beta+j}{2}\nonumber\\
	=\; & \frac{2(k_1-1)}{k_1}\sum_{\alpha = 2n}^{4n}\sum_{\beta = 2n}^{4n} c_{\alpha} c_{\beta} \nonumber\\
	=\; &\frac{2(k_1-1)}{k_1} \; > \; 0
	\end{align} where we have used $\sum_{i = 0}^{2}\binom{2}{i}(-1)^{2-i}\binom{\alpha+i}{2} = \sum_{j = 0}^{2} \binom{2}{j}(-1)^{2-j}\binom{\beta+j}{2} = 1$.
\end{exa}

Therefore, the traditional way of showing weak convergence and absolute convergence fail here. In order to resolve this, we modify the empirical blip spectral measure by taking average over a large number of matrices and prove that the modified blip spectral measure converges. The definitions and the process of the proof follow closely to Section 5 of \cite{BCDHMSTPY}. The only change needed is the proof of the following lemma.

\begin{lemma}\label{lem_moments_of_moments}
Let $X_{m,N,i}$ be as defined in Definition \ref{defi of X}
Then for any $t \in \N$, the $r$\textsuperscript{th} centered moment of $X_{m,N,i}$ satisfies
\begin{equation}
X_{m,N,i}^{(r)} \ =\ \E\left[\left(X_{m,N,i}-\E[X_{m,N,i}]\right)^r\right]\ =\ O_{m,r}(1)
\end{equation}
as $N$ goes to infinity.
\end{lemma}

The proof of this lemma uses similar technique as Lemma 5.6 of \cite{BCDHMSTPY} and reduces to a special case of Section \ref{sec:blip}. The detail of the proof is also given in Appendix \ref{appendix F}.


\appendix

\section{Proof of Lemma \ref{lem:poly}} \label{sec:appendixproof1}
\begin{lemma} \label{lem_a3}
    Fix $s\in\mathbb{N}$ with $s\geq 2$. For $\eta\in\mathbb{N}$ with $\eta\geq s$ and distinct $w_1,\dots, w_s$, we have \begin{align} \label{732}
	\sum_{\substack{x_1+\dots + x_s = \eta\\ x_1,\dots, x_s\geq 1}} w_1^{x_1} \dots w_{s}^{x_s} \; = \; \frac{\sum_{l = 1}^{s} w_{l}^{\eta+2-s} f_{l}(w_1,\dots, w_s)}{\prod_{1\leq i < j \leq s}(w_i - w_j)},
	\end{align} where each $f_{l}(w_1,\dots, w_s) \in \mathbb{R}[w_1,\dots, w_s]$ is a homogeneous polynomial of degree $\binom{s}{2} + s-2$.
\end{lemma}

\begin{proof}
	Induct on $s$. For $s = 2$, by geometric progression, we have \begin{align}
	\sum_{\substack{x_1+x_2 = \eta\\ x_1, x_2 \geq 1}}w_1^{x_1}w_2^{x_2}  \; =\; \frac{w_1^{\eta}w_2 - w_2^{\eta}w_1}{w_1-w_2}
	\end{align} for all $\eta\in\mathbb{N}$ with $\eta\geq 2$. Suppose $s\in\mathbb{N}$ with $s\geq 2$ and equation \eqref{732} holds for all $\eta\in\mathbb{N}$ with $\eta\geq s$.
	Then for $\eta \geq s+1$, \begin{align}
		\sum_{\substack{x_1+\dots + x_{s+1} = \eta\\ x_1,\dots, x_{s+1}\geq 1}} w_1^{x_1} \dots w_{s+1}^{x_{s+1}} \; = \;& \sum_{x_{s+1} = 1}^{\eta-s} \sum_{\substack{x_1+\dots + x_s = \eta-x_{s+1}\\ x_1,\dots, x_s\geq 1}} w_1^{x_1} \dots w_{s}^{x_s}w_{s+1}^{x_{s+1}}\nonumber\\
		=\; &\sum_{x_{s+1} = 1}^{\eta-s} \frac{\sum_{l = 1}^{s} w_{l}^{\eta-x_{s+1}+2-s} f_{l}(w_1,\dots, w_s)}{\prod_{1\leq i < j \leq s}(w_i - w_j)}w_{s+1}^{x_{s+1}} \nonumber\\
		=\;& \sum_{l =1}^{s}\frac{f_{l}(w_1,\dots, w_s)}{\prod_{1\leq i < j \leq s}(w_i - w_j)}\sum_{x_{s+1} = 1}^{\eta-s}w_{l}^{\eta-x_{s+1}+2-s}w_{s+1}^{x_{s+1}} \nonumber\\
		=\; &\sum_{l =1}^{s}\frac{f_{l}(w_1,\dots, w_s)}{\prod_{1\leq i < j \leq s}(w_i - w_j)}\frac{w_{l}^{\eta+1-s}w_{s+1} - w_{l}w_{s+1}^{\eta+1-s} }{1-\frac{w_{s+1}}{w_l}} \nonumber\\
		=\; & \sum_{l =1}^{s}\frac{w_{l}f_{l}(w_1,\dots, w_s)}{\prod_{1\leq i < j \leq s}(w_i - w_j)}\frac{w_{l}^{\eta+1-s}w_{s+1} - w_{l}w_{s+1}^{\eta+1-s} }{w_{l}-w_{s+1}} \nonumber\\
		=\; & \frac{\sum_{l=1}^{s}w_{l}^{\eta+1-s}(w_{l}w_{s+1}\prod_{\substack{1\leq i  \leq s\\ i\neq l}}(w_i - w_{s+1}) f_{l}(w_1,\dots, w_s) )  }{\prod_{1\leq i < j \leq s+1}(w_i - w_j)} \nonumber\\
		&-\frac{w_{s+1}^{\eta+1-s}(\sum_{l=1}^{s}w_{l}^{2}\prod_{\substack{1\leq i  \leq s\\ i\neq l}}(w_i - w_{s+1}) f_{l}(w_1,\dots, w_s) )  }{\prod_{1\leq i < j \leq s+1}(w_i - w_j)},
	\end{align} where each $w_{l}w_{j}\prod_{\substack{1\leq i  \leq s\\ i\neq l}}(w_i - w_{s+1}) f_{l}(w_1,\dots, w_s)$ is a homogeneous polynomial in $w_1,\dots, w_{s+1}$ of degree $2+(s-1)+(\binom{s}{2}+s-2) \; = \; \binom{s+1}{2}+s+1-2$.
\end{proof}

\begin{lemma}\label{a4}
	Fix $s\in\mathbb{N}$ with $s\geq 2$, $q\in\mathbb{N}\cup\{0\}$, and $\alpha_1,\dots, \alpha_q\in\mathbb{N}_{\leq s}$ (may not be distinct). For $\eta\in\mathbb{N}$ with $\eta\geq s$ and distinct $w_1,\dots, w_s$, we have \begin{align} \label{735old}
	\sum_{\substack{x_1+\dots + x_s = \eta\\ x_1,\dots, x_s\geq 1}}x_{\alpha_1}\dots x_{\alpha_q} w_1^{x_1} \dots w_{s}^{x_s} \; = \; \frac{\sum_{l = 1}^{s} w_{l}^{\eta+2-s} f_{l,\eta}(w_1,\dots, w_s)}{(\prod_{1\leq i < j \leq s}(w_i - w_j))^{2^{q}}},
	\end{align} where each $f_{l,\eta}(w_1,\dots, w_s)\in\mathbb{R}[\eta][w_1,\dots, w_s]$ is a homogeneous polynomial in $w_1,\dots, w_s$ of degree $2^{q}\binom{s}{2} + s-2$. Furthermore, the coefficients in the polynomial $f_{l,\eta}(w_1,\dots, w_s)$ are polynomials in $\eta$ of degree $\leq q$.
\end{lemma}
\begin{proof}
	Induct on $q$. The case $q = 0$ was proved in Lemma \ref{lem_a3}.
	Suppose $q\in\mathbb{N}$ and we have \begin{align}
	\sum_{\substack{x_1+\dots + x_s = \eta\\ x_1,\dots, x_s\geq 1}}x_{\alpha_1}\dots x_{\alpha_{q-1}} w_1^{x_1} \dots w_{s}^{x_s} \; = \; \frac{\sum_{l = 1}^{s} w_{l}^{\eta+2-s} f_{l,\eta, \alpha_1,\dots, \alpha_{q-1}}(w_1,\dots, w_s)}{(\prod_{1\leq i < j \leq s}(w_i - w_j))^{2^{q-1}}}
	\end{align} for all $\eta\in\mathbb{N}$ with $\eta \geq s$, where each $f_{l,\eta, \alpha_1,\dots, \alpha_{q-1}}(w_1,\dots, w_s)$ is a homogeneous polynomial in $w_1,\dots, w_s$ of degree $2^{q-1}\binom{s}{2} + s-2$ and the coefficients are polynomials in $\eta$ of degree $\leq q-1$. Then \begin{align}
	&\sum_{\substack{x_1+\dots + x_s = \eta\\ x_1,\dots, x_s\geq 1}}x_{\alpha_1}\dots x_{\alpha_q} w_1^{x_1} \dots w_{s}^{x_s} \nonumber\\
	=\; &w_{\alpha_q}\frac{\partial}{\partial w_{\alpha_q}}\sum_{\substack{x_1+\dots + x_s = \eta\\ x_1,\dots, x_s\geq 1}}x_{\alpha_1}\dots x_{\alpha_{q-1}} w_1^{x_1} \dots w_{s}^{x_s} \nonumber\\
	=\; & w_{\alpha_q}\frac{\partial}{\partial w_{\alpha_q}}\frac{\sum_{l = 1}^{s} w_{l}^{\eta+2-s} f_{l,\eta, \alpha_1,\dots, \alpha_{q-1}}(w_1,\dots, w_s)}{(\prod_{1\leq i < j \leq s}(w_i - w_j))^{2^{q-1}}} \nonumber\\
	=\; & \frac{\sum_{l = 1}^{s} w_{l}^{\eta+2-s} w_{\alpha_q}  \frac{\partial f_{l,\eta,\alpha_1,\dots, \alpha_{q-1}}(w_1,\dots, w_s)}{\partial w_{\alpha_q}}(\prod_{1\leq i < j \leq s}(w_i - w_j))^{2^{q-1}}}{(\prod_{1\leq i < j \leq s}(w_i - w_j))^{2^{q}}} \nonumber\\&
	+ \frac{w_{\alpha_q}^{\eta+2-s}(\eta+2-s)f_{l,\eta,\alpha_1,\dots, \alpha_{q-1}}(w_1,\dots, w_s)(\prod_{1\leq i < j \leq s}(w_i - w_j))^{2^{q-1}}}{(\prod_{1\leq i < j \leq s}(w_i - w_j))^{2^{q}}} \nonumber\\&
	- \frac{\sum_{l = 1}^{s} w_{l}^{\eta+2-s} f_{l,\eta,\alpha_1,\dots, \alpha_{q-1}}(w_1,\dots, w_s)w_{\alpha_q}  \frac{\partial}{\partial w_{\alpha_q}}(\prod_{1\leq i < j \leq s}(w_i - w_j))^{2^{q-1}}}{(\prod_{1\leq i < j \leq s}(w_i - w_j))^{2^{q}}}.
	\end{align}
	Note that by induction hypothesis, we have \begin{enumerate} \item $w_{\alpha_q}  \frac{\partial f_{l,\eta,\alpha_1,\dots, \alpha_{q-1}}(w_1,\dots, w_s)}{\partial w_{\alpha_q}}(\prod_{1\leq i < j \leq s}(w_i - w_j))^{2^{q-1}}$ is a homogeneous polynomial in $w_1,\dots, w_s$ of degree $1+(2^{q-1}\binom{s}{2} + s-2-1)+2^{q-1}\binom{s}{2} =2^{q}\binom{s}{2} + s-2 $ and the coefficients are polynomials in $\eta$ of degree $\leq q-1$;
	\item  $(\eta+2-s)f_{l,\eta,\alpha_1,\dots, \alpha_{q-1}}(w_1,\dots, w_s)(\prod_{1\leq i < j \leq s}(w_i - w_j))^{2^{q-1}}$ is a homogeneous polynomial in $w_1,\dots, w_s$ of degree $(2^{q-1}\binom{s}{2} + s-2)+2^{q-1}\binom{s}{2} =2^{q}\binom{s}{2} + s-2 $ and the coefficients are polynomials in $\eta$ of degree $\leq q$;
	\item $f_{l,\eta,\alpha_1,\dots, \alpha_{q-1}}(w_1,\dots, w_s)w_{\alpha_q}  \frac{\partial}{\partial w_{\alpha_q}}(\prod_{1\leq i < j \leq s}(w_i - w_j))^{2^{q-1}}$ is a homogeneous polynomial in $w_1,\dots, w_s$ of degree $(2^{q-1}\binom{s}{2} + s-2)+1+(2^{q-1}\binom{s}{2}-1) =2^{q}\binom{s}{2} + s-2 $ and the coefficients are polynomials in $\eta$ of degree $\leq q-1$.
\end{enumerate}
	Therefore, after collecting the terms, we get \eqref{735old}.
\end{proof}

\begin{lemma} \label{a5}
    Fix $s\in\mathbb{N}$ with $s\geq 2$ and some polynomial $p(x_1,\dots, x_s)\in\mathbb{R}[x_1,\dots, x_s]$ of degree $q$. For $\eta\in\mathbb{N}$ with $\eta\geq s$ and distinct $w_1,\dots, w_s$, we have \begin{align}
        \sum_{\substack{x_1 + \dots + x_s = \eta \\ x_1,\dots, x_s \geq 1}} p(x_1,\dots, x_s) w_1^{x_1}\cdots w_s^{x_s} \; = \; \frac{\sum_{l =1}^{s}w_{l}^{\eta + 2 -s}f_{l,\eta}(w_1,\dots, w_s)}{(\prod_{1\leq i < j \leq s}(w_i - w_j))^{2^{q}}},
    \end{align} where each $f_{l,\eta}(w_1,\dots, w_s)\in\mathbb{R}[\eta][w_1,\dots, w_s]$ is a homogeneous polynomial in $w_1,\dots, w_s$ of degree $2^q\binom{s}{2} + s - 2$. Furthermore, the coefficients in the polynomial $f_{l,\eta}(w_1,\dots, w_s)$ are polynomials in $\eta$ of degree $\leq q$.
\end{lemma}

\begin{proof}
By Lemma \ref{a4}, fix any $d\in\mathbb{N}\cup \{0\}$ with $d\leq q$, and $\alpha_1, \dots, \alpha_d\in\mathbb{N}_{\leq s}$, we have \begin{align}
	\sum_{\substack{x_1+\dots + x_s = \eta\\ x_1,\dots, x_s\geq 1}}x_{\alpha_1}\dots x_{\alpha_d} w_1^{x_1} \dots w_{s}^{x_s} \; = \; &\frac{\sum_{l = 1}^{s} w_{l}^{\eta+2-s} \tilde{f}_{l,\eta}(w_1,\dots, w_s)}{(\prod_{1\leq i < j \leq s}(w_i - w_j))^{2^{d}}}\nonumber\\
	=\; &\frac{\sum_{l = 1}^{s} w_{l}^{\eta+2-s} \tilde{f}_{l,\eta}(w_1,\dots, w_s)(\prod_{1\leq i < j \leq s}(w_i - w_j))^{2^{q} - 2^{d}}}{(\prod_{1\leq i < j \leq s}(w_i - w_j))^{2^{q}}},
	\end{align} for some degree $2^{d}\binom{s}{2} + s-2 $ homogeneous polynomials $\tilde{f}_{l,\eta}(w_1,\dots, w_s)\in\mathbb{R}[\eta][w_1,\dots, w_s]$ whose coefficients are polynomials in $\eta$ of degree $\leq d$. Then $\tilde{f}_{l,\eta}(w_1,\dots, w_s)(\prod_{1\leq i < j \leq s}(w_i - w_j))^{2^{q} - 2^{d}}$ are homogeneous polynomials in $w_1,\dots, w_s$ of degree $2^{d}\binom{s}{2} + s-2+(2^{q} - 2^{d})\binom{s}{2} = 2^{q}\binom{s}{2} + s-2 $, and the coefficients are polynomials in $\eta$ of degree $\leq d\leq q$, and the result follows.
\end{proof}

\begin{lemma} \label{a6}
    Fix $s\in\mathbb{N}$ with $s\geq 2$, $q\in\mathbb{N}\cup\{0\}$, $y_1,\dots, y_s\in\mathbb{N}$, and $\alpha_1,\dots, \alpha_q\in\mathbb{N}_{\leq s}$ (may not be distinct). For $\eta\in\mathbb{N}$ with $\eta\geq \sum_{i=1}^{s} y_i$ and distinct $w_1,\dots, w_s$, we have \begin{align} \label{735}
	\sum_{\substack{x_1+\dots + x_s = \eta\\ x_i\geq y_i}}x_{\alpha_1}\dots x_{\alpha_q} w_1^{x_1} \dots w_{s}^{x_s} \; = \; \frac{\sum_{l = 1}^{s} w_{l}^{\eta + 2 -\sum_{i=1}^{s}y_i} f_{l,\eta}(w_1,\dots, w_s)}{(\prod_{1\leq i < j \leq s}(w_i - w_j))^{2^{q}}},
	\end{align} where each $f_{l,\eta}(w_1,\dots, w_s)\in\mathbb{R}[\eta][w_1,\dots, w_s]$ is a homogeneous polynomial in $w_1,\dots, w_s$ of degree $ 2^{q}\binom{s}{2} + (\sum_{i=1}^{s}y_i)-2$. Furthermore, the coefficients in the polynomial $f_{l,\eta}(w_1,\dots, w_s)$ are polynomials in $\eta$ of degree $\leq q$.
\end{lemma}
\begin{proof}
By lemma \ref{a5}, we have \begin{align}
    &\sum_{\substack{ x_1+\dots+x_s=\eta \\ x_i \geq y_i}}
x_{\alpha_1}\dots x_{\alpha_q}w_{1}^{x_1}\dots w_{s}^{x_s} \; \nonumber\\= \;& \prod_{i = 1}^s w_{i}^{y_i-1}\sum_{\substack{ x_1+\dots+x_s=\eta+s-\sum_{i = 1}^{s}y_i \\ x_1,\dots, x_s \geq 1}}
(x_{\alpha_1}+y_{\alpha_1}-1)\cdots (x_{\alpha_q}+y_{\alpha_q}-1)w_{1}^{x_1}\dots w_{s}^{x_s} \nonumber\\
=\; & \prod_{i = 1}^s w_{i}^{y_i-1}\frac{\sum_{l =1}^{s}w_{l}^{\eta + 2 -\sum_{i=1}^{s}y_i }\tilde{f}_{l,\eta}(w_1,\dots, w_s)}{(\prod_{1\leq i < j \leq s}(w_i - w_j))^{2^{q}}},
\end{align} where each $\tilde{f}_{l,\eta}(w_1,\dots, w_s)\in\mathbb{R}[\eta][w_1,\dots, w_s]$ is a homogeneous polynomial in $w_1,\dots, w_s$ of degree $2^{q}\binom{s}{2} + s-2$, and the coefficients of $f_{l,\eta}(w_1,\dots, w_s)$ are polynomials in $\eta$ of degree $\leq q$. Then each $f_{l,\eta}(w_1,\dots, w_s) = \prod_{i = 1}^s w_{i}^{y_i-1}\tilde{f}_{l,\eta}(w_1,\dots, w_s)$ is a homogeneous polynomial in $w_1,\dots, w_s$ of degree $ 2^{q}\binom{s}{2} + (\sum_{i=1}^{s}y_i)-2$, and the coefficients are polynomials in $\eta$ of degree $\leq q$.
\end{proof}

\begin{lemma}
	Fix $s\in\mathbb{N}$ with $s\geq 2$ and some polynomial $p(x_1,\dots, x_s)\in\mathbb{R}[x_1,\dots, x_s]$ of degree $q$. For $\eta\in\mathbb{N}$ with $\eta\geq \sum_{i=1}^{s} y_i$ and distinct $w_1,\dots, w_s$, we have \begin{align}
	\sum_{\substack{x_1+\dots + x_s = \eta\\ x_i\geq y_i}}p(x_1,\dots, x_s) w_1^{x_1} \dots w_{s}^{x_s} \; = \; \frac{\sum_{l = 1}^{s} w_{l}^{\eta + 2 -\sum_{i=1}^{s}y_i} f_{l,\eta}(w_1,\dots, w_s)}{(\prod_{1\leq i < j \leq s}(w_i - w_j))^{2^{q}}},
	\end{align} where each $f_{l,\eta}(w_1,\dots, w_s)\in\mathbb{R}[\eta][w_1,\dots, w_s]$ is a homogeneous polynomial in $w_1,\dots, w_s$ of degree $ 2^{q}\binom{s}{2} + (\sum_{i=1}^{s}y_i)-2$. Furthermore, the coefficients in the polynomial $f_{l,\eta}(w_1,\dots, w_s)$ are polynomials in $\eta$ of degree $\leq q$.
\end{lemma}
\begin{proof}
	By Lemma \ref{a6}, fix any $d\in\mathbb{N}\cup \{0\}$ with $d\leq q$, and $\alpha_1, \dots, \alpha_d\in\mathbb{N}_{\leq s}$, we have \begin{align}
	\sum_{\substack{x_1+\dots + x_s = \eta\\ x_i\geq y_i}}x_{\alpha_1}\dots x_{\alpha_d} w_1^{x_1} \dots w_{s}^{x_s} \; =\ \; &\frac{\sum_{l = 1}^{s} w_{l}^{\eta+2-s} \tilde{f}_{l,\eta}(w_1,\dots, w_s)}{(\prod_{1\leq i < j \leq s}(w_i - w_j))^{2^{d}}}\nonumber\\
	=\ \; &\frac{\sum_{l = 1}^{s} w_{l}^{\eta+2-s} \tilde{f}_{l,\eta}(w_1,\dots, w_s)(\prod_{1\leq i < j \leq s}(w_i - w_j))^{2^{q} - 2^{d}}}{(\prod_{1\leq i < j \leq s}(w_i - w_j))^{2^{q}}},
	\end{align} for some degree $2^{d}\binom{s}{2} + (\sum_{i=1}^{s}y_i)-2 $ homogeneous polynomials $\tilde{f}_{l,\eta}(w_1,\dots, w_s)\in\mathbb{R}[\eta][w_1,\dots, w_s]$ whose coefficients are polynomials in $\eta$ of degree $\leq d$. Then $\tilde{f}_{l,\eta}(w_1,\dots, w_s)(\prod_{1\leq i < j \leq s}(w_i - w_j))^{2^{q} - 2^{d}}$ are homogeneous polynomials in $w_1,\dots, w_s$ of degree $2^{d}\binom{s}{2} + (\sum_{i=1}^{s}y_i)-2+(2^{q} - 2^{d})\binom{s}{2} = 2^{q}\binom{s}{2} + (\sum_{i=1}^{s}y_i)-2 $, and the coefficents are polynomials in $\eta$ of degree $\leq d\leq q$, and the result follows.
\end{proof}

\section{Proof of Lemma \ref{lem_zero}} \label{sec:appendixproof2}
\begin{lemma} \label{lem_zero1}
Suppose the polynomial $f(x)  :=  \sum_{\alpha} c_{\alpha} x^{\alpha} \in \mathbb{R}[x]$ has a zero of order $n > 0$ at $x_0$. Then for $d\in\mathbb{N}\cup \{0\}$ with $d < n$, the polynomial $f_{d}(x)  := \sum_{\alpha}c_{\alpha}x^{\alpha} \alpha^{d}$ has a zero of order $(n-d)$ at $x_0$. In particular, we have \begin{align}
    \sum_{\alpha}c_{\alpha}x_0^{\alpha} \alpha^{d} \; = \; 0
\end{align} for all $d\in\mathbb{N}\cup \{0\}$ with $d < n$.
\end{lemma}

\begin{proof}
The result follows from the fact that \begin{align}f_{d}(x) = xf_{d-1}'(x)\end{align} for all $d\in\mathbb{N}$ with $d < n$.
\end{proof}

\begin{lemma}
    Suppose the polynomial $f(x)  :=  \sum_{\alpha} c_{\alpha} x^{\alpha} \in \mathbb{R}[x]$ has a zero of order $n > 0$ at $x_0$. Then \begin{align}
    \sum_{\alpha}c_{\alpha}x_0^{\alpha} p(\alpha) \; = \; 0
\end{align} for all polynomial $p$ of degree $d < n$.
\end{lemma}

\begin{proof}
The result follows immediately from Lemma \ref{lem_zero1}.
\end{proof}

\section{Proof of Proposition \ref{proposition contribution}}  \label{app:blip}
By Lemma \ref{lem_contributions}, the configurations with the highest number of degrees of freedom contain only $1$- and $2$-blocks.
The number of ways to arrange the constants $w_1$'s and the blocks (take all blocks to be identical) is
\begin{equation*}
    \frac{(\alpha+i)^{\beta}}{\beta!}+\tilde{p}(\alpha+i)
\end{equation*} where $\tilde{p}$ is a polynomial of degree $\leq \beta-1$,
and the number of ways to choose the 1-blocks among all the blocks is $\binom{\beta}{m_1}$.

Now we assign the equivalence classes modulo $k$ of the inner indices of the 2-blocks. The number of ways to assign inner indices of 2-blocks is $(k-1)^{\beta-m_1}$. The number of ways to assign indices of the 1-blocks is the same as the number of cyclic product
$b_{i_1i_2}b_{i_2i_3}\cdots b_{i_{m_1}i_1}$, where $i_j$'s are chosen from $k_1$ residues modulo $k$ with the $b$'s matched in pairs under the restriction that $i_{j} \neq i_{j+1}$ for all $j$. Thus it is the expected trace of $m_1^{\text{th}}$ power of $k_1 \times k_1$ GOE, which is
\begin{equation*}
    \mathbb{E}_{k_1}\left[\text{Tr}\left(B^{m_1}\right)\right].
\end{equation*}
Finally, for each index, once we have specified its congruence class modulo $k$, the number of ways to choose it from $\{1,2,\dots, N\}$ is $\left(\frac{N}{k}\right)^{\alpha+i-\beta}+O(\frac{N}{k})^{\alpha+i-\beta-1}$.

\section{Proof of Proposition \ref{prop:expmoment}} \label{app:blip2}
By Proposition $\ref{proposition contribution}$ and \eqref{eqn:blipmoment}, we get the contribution from the configurations with $\beta$ blocks to the expected $m^{\text{th}}$ moment of the blip is
\begin{align} \label{eqn:blipmomentcontribution}
&\frac{1}{k_1}\sum_{\alpha = 2n} ^{4nl}c_{\alpha} \left(\frac{k}{w_1N}\right)^{\alpha}\sum_{i = 0}^{m}\binom{m}{i} \left(-\frac{w_1N}{k}\right)^{m-i} \nonumber\\
&\left(\sum_{m_1 = 0}^{\beta}
w_1^{\alpha+i-m_1-2(\beta-m_1)}
\binom{\beta}{m_1}
(k-1)^{\beta-m_1}
\mathbb{E}_{k_1}\left[\text{Tr}(B^{m_1})\right]\right) \nonumber\\
&\left(\frac{(\alpha+i)^{\beta}}{\beta!}+\tilde{p}(\alpha+i)\right)
\left(\frac{N}{k}\right)^{\alpha+i-\beta}
+O_{\beta}\left((\alpha+i)^{\beta}
\left(\frac{N}{k}\right)^{\alpha+i-\beta-1}
\right).
\end{align}

Recall that by \eqref{equation 3.5} and Lemma \ref{lem_contributions}, the contribution becomes insignificant as $N \to \infty$ if $\beta>m$.
On the other hand, given any polynomial $p(x)\in\mathbb{R}[x]$ of degree less than $m$ and $t\in\mathbb{Z}$, we have $\sum_{i = 0}^{m}\binom{m}{i} \left(-\frac{w_1N}{k}\right)^{m-i}p(\alpha+i) \left(\frac{N}{k}\right)^{\alpha+i-t}= \left(\frac{N}{k} \right)^{m+\alpha-t} \sum_{i = 0}^{m}\binom{m}{i}(-1)^{m-i}p(\alpha+i) = 0$ from Lemma \ref{lem_zero} using the fact that $(x-1)^m$ has a zero of order $m$ at $1$, so the contribution cancels out if $\beta < m$.
Therefore, only the configurations with $m$ blocks will contribute.

We set $\beta=m$ in \eqref{eqn:blipmomentcontribution}, and use the identity $$\sum_{i = 0}^{m} \binom{m}{i} (-1)^{m-i}i^{j} \;=\; \begin{cases}
 0 & \text{ if } j = 0,1,\dots, m-1, \\
 m! & \text{ if } j = m,
\end{cases}$$ and the fact that $ \sum_{\alpha = 2n}^{4nl} c_{\alpha} = f_1^{2n}(1) = 1$ to get the expected $m$-th  moment
\begin{equation}
    \lim\limits_{N\to\infty}\mathbb{E}\left[\mu_{A_N,1}^{(m)}\right] \; = \;\frac{1}{k_1}\sum_{m_1=0}^m\binom{m}{m_1}\left(\frac{k-1}{w_1}\right)^{m-m_1}\mathbb{E}_{k_1}\left[\text{Tr}(B^{m_1})\right].
\end{equation}
\section{Details for absolute convergence and Proof of Lemma 4.6}\label{appendix F}


By \cite{BCDHMSTPY}, we can treat the $m\textsuperscript{th}$ moment of empirical spectral measure near $Nw_i/k$, $\mu_{A_N,i}^{(m)}$, as a random variable on $\Omega$. Here $\Omega\ :=\ \prod_{N \in \N} \Omega_N$, where $\Omega_N$ is the probability space of $N \times N$ $(k,W)$ Checkerboard matrices.
\begin{defi}\cite{BCDHMSTPY}\label{defi of X}
We define the random variable $X_{m,N,1}$ on $\Omega$
\begin{equation}\label{eq_xmn}
X_{m,N,i}(\{A_N\})\ :=\ \mu_{A_N,i}^{(m)}.
\end{equation}
\end{defi}
which has the centered $r$\textsuperscript{th} moment as
\begin{equation}
X_{m,N,i}^{(r)}\ :=\ \E[(X_{m,N,i}-\E[X_{m,N,i}])^r].
\end{equation}

\begin{defi}\cite{BCDHMSTPY}\label{def_average_blip_measure}
Fix a function $g: \N \rightarrow \N$. The \textbf{averaged empirical blip spectral measure} associated to $\overline{A} \in \Omega^\N$ is
\begin{equation}
\mu_{N,g,\overline{A},i}\ :=\ \frac{1}{g(N)}\sum_{j=1}^{g(N)} \mu_{A_N^{(j)},i}
\end{equation}
\end{defi}

This is to project onto the $N$\textsuperscript{th} coordinate in each copy of $\Omega$ and then average over the first $g(N)$ of these $N \times N$ matrices.

\begin{defi}\cite{BCDHMSTPY}
 We denote by $Y_{m,N,g,i}$ the random variable on $\Omega^\N$ defined by the moments of the averaged empirical blip spectral measure near $Nw_i/k$.
\begin{equation}
Y_{m,N,g,i}(\overline{A})\ :=\ \mu_{N,g,\overline{A},i}^{(m)}.
\end{equation}
The centered $r$\textsuperscript{th} moment (over $\Omega^\N$) of this random variable will be denoted by $Y_{m,N,g,i}^{(r)}$.
\end{defi}


With the defintions, we are ready to prove Lemma \ref{lem_moments_of_moments}.


\begin{proof}
Without loss of generality it suffices to prove it when $i=1$.
 \begin{align}
 \E\left[\left(X_{m,N,1}-\E[X_{m,N,1}]\right)^r\right]\ &\ =\  \E\left[\sum_{l=0}^r \binom{r}{l}(X_{m,N,1})^\ell \left(\E[X_{m,N,1}]\right)^{r-l}\right]\  \nonumber \\
 &\ =\ \ \sum_{l =0}^r \binom{r}{l}(-1)^l\E\left[(X_{m,N,1})^l\right] \left(\E[X_{m,N,1}]\right)^{r-l}.
\end{align}
By \eqref{moment}, we have $\E[X_{m,N,1}]=O_m(1)$, hence $\left(\E[X_{m,N,1}]\right)^{r-l}=O_{m,r,l}(1)$ for all $l$. As such, it suffices to show that $\E\left[(X_{m,N,1})^\l\right]=O_{m,l}(1)$.

\begin{align}
\E[{X_{m,N,1}}^l]\ &\ =\ \E\left[ \frac{1}{k_1}\sum_{\alpha = 2n} ^{4nl}c_{\alpha} \left(\frac{k}{w_1N}\right)^{\alpha}\left(\sum_{i = 0}^{m}\binom{m}{i} \left(-\frac{w_1N}{k}\right)^{m-i}\text{Tr}(A_N^{\alpha+i})\right)^l\right]  \nonumber \\
&\ =\  \E\left[\sum_{\substack{2n \leq \ga_1 \leq 4nl\\ 0 \leq i_1 \leq m}} \sum_{\substack{2n \leq \ga_2 \leq 4nl\\ 0 \leq i_2 \leq m}}\dots\sum_{\substack{2n \leq \ga_l \leq 4nl\\ 0 \leq i_l \leq m}} \prod_{\nu=1}^lc_{\ga_v} {m \choose i_v} (-1)^{m-i_v}(\frac{w_1N}{k})^{l-i_v}
 \tr A^{\ga_+i_v} \right] \nonumber \\
&\ =\  \sum_{\substack{2n \leq \ga_1 \leq 4nl\\ 0 \leq i_1 \leq m}} \sum_{\substack{2n \leq \ga_2 \leq 4nl\\ 0 \leq i_2 \leq m}}\dots\sum_{\substack{2n \leq \ga_l \leq 4nl\\ 0 \leq i_l \leq m}} \prod_{\nu=1}^lc_{\ga_v} {m \choose i_v} (-1)^{m-i_v}(\frac{w_1N}{k})^{l-i_v}
\E \left[\prod_{v=1}^l \tr A^{\ga_v+i_v} \right]. \label{eq_finite_moments}
\end{align}

Now consider $\E \left[\prod_{v=1}^l \tr A^{\ga_v+i_v} \right]$. Let $C_t$ denotes
 the cycle $m_{i_1i_2}m_{i_2i_3}\dots m_{i_ti_1}$, then
 \[\E \left[\prod_{v=1}^l \tr A^{\ga_v+i_v} \right]=\sum_{C_{\alpha_1+i_1}}\sum_{C_{\ga_2+i_2}} \dots \sum_{C_{\ga_l+i_l}}\E\left[\prod_{v=1}^l C_{\ga_v+i_v} \right],\]
where $\sum_{C_{\ga_k+i_k}}$ means summing over all cycles of length $\ga_k+i_k$.

Now fix the congruence class of each cycle. For a fixed congruence configuration of $\tr A^{\ga_v+i_v}$, suppose that the distinct constants $w_{j_1}^v$, $w_{j_2}^v$, $\dots$, $w_{j_{s_{v}}}^v$ appear in the configuration, where each $w_{j_q}^v$ appears $x_{q}^v$ times and the $x_q^v$ $w_{j_q}^v$'s are seperated by the blocks into $y_q^v$ parts. Notice that when we fix the configurations, $y_q^v$ are fix.  Denote the number of $a$ as $r_v$, so we have $x_1^v+\dots +x_{s_v}^v=\ga_v+i_v-r_v$. When we fix the configuration for all the cycles, by Equation $\eqref{eqn:combinatorics}$ there are
\begin{equation}
    \prod_{v=1}^l\sum_{\substack{x_i^v \geq y_i^v \\ x_1^v+\dots+x_s^v=\alpha_v+i_v-r_v}}\prod_{q_v=1}^{s_v}\binom{x_q^v-1}{y_q^v-1}w_{j_q^v}^{x_q^v}\end{equation}
    ways to place the constants $w_{j_1}^v$, $w_{j_2}^v$, $\dots$, $w_{j_s}^v$.
    For fixed $y_1^v,\dots, y_{s_v}^v$, we can write it as \begin{equation}
    \prod_{v=1}^l\sum_{\substack{x_i^v \geq y_i^v \\ x_1^v+\dots+x_{s_v}^v=\alpha_v+i_v-r_v}}\widetilde{g_v}_{y_1^v,\dots, y_{s_v}^v} (x_1^v,\dots, x_{s_v}^v)w_{j_1^v}^{x_1^v}w_{j_2^v}^{x_2^v}\cdots w_{j_{s_v}^v}^{x_{s_v}^v}\end{equation}
    where $\widetilde{g_v}_{y_1^v,\dots, y_{s_v}^v}(x_1^v,\dots, x_{s_v}^v) \in \mathbb{R}[x_1^v,\dots, x_{s_v}^v]$ is a polynomial in $x_1^v,\dots, x_{s_v}^v$ of degree $\sum_{q_v=1}^{s_v}(y_{q_v}^v-1) = (y_1^v+\cdots+y_{s_v}^v) - s_v = \beta_v-s_v$.
Following the same reasoning as Lemma $\ref{lem:sumform}$, apply Lemma $\ref{lem:poly}$ and take into account the cyclic permutation, we get that the total contribution can be written as a sum of the terms of the form
\begin{equation*}
    \prod_{v=1}^lp_v(\alpha_v+i_v)w_{j_v}^{\alpha_v+i_v-\gamma_v}\left(\frac{N}{k}\right)^{\alpha_v+i_v-t_v}
    \end{equation*}
    where $p(x) \in\mathbb{R}[x]$ is a polynomial of degree $\leq \beta_v-s_v+1$, $\gamma\in\mathbb{Z}$ and $\left(\frac{N}{k}\right)^{\alpha+i-t}$ is from choosing the indices from given equivalence classes modulo $k$, and $t_v$ is the lost degree of freedom.

    By similar reasoning as \eqref{equation 3.5}, we can restrict ourselves to the set of configurations that $\beta_1+\beta_2+\dots+\beta_l \leq ml $.
    If any of the $\beta_v \leq m-1$, or $s_v \geq 2$ (which means there is some $w_{j_v} \neq w_1$ in the configuration), then the degree of $p_v(\ga_v+i_v) \leq m-1$. By $\eqref{eqn:eqn1}$  and $\eqref{eqn:eqn2}$,
    \[\sum_{\alpha_v = 2n} ^{4nl}c_{\alpha_v} \left(\frac{k}{w_1N}\right)^{\alpha_v}\left(\sum_{i = 0}^{m}\binom{m}{i_v} \left(-\frac{w_1N}{k}\right)^{m-i_v}p_v(\alpha_v+i_v) w_{jv}^{\alpha_v + i_v-\gamma_v}\left(\frac{N}{k}\right)^{\alpha_v + i_v - t_v}\right) \ = \ 0.\]
Then the contribution of the configuration to $\E \left[\prod_{v=1}^l \tr A^{\ga_v+i_v} \right]$ is $0$. Therefore, the only set of configurations that make contribution to  $\E \left[\prod_{v=1}^l \tr A^{\ga_v+i_v} \right]$ is those where $\beta_v=m$ for all $v$, and only $w_1$ between the blocks, and there are only 1-blocks and 2-blocks that match with each other in the way described by Lemma $\ref{lem_contributions}$.

Now we can apply the same argument as in the Lemma 5.6 of \cite{BCDHMSTPY}. At each cycle, fix the number of blocks $\beta_v=m$, the number of ways to arrange the blocks and $w_1$'s is $\left(\frac{(\alpha_v+i_v)^{\beta_v}}{\beta_v!}+\tilde{p}(\alpha_v+i_v)\right)$ where $\tilde{p}(\alpha_v+i_v)$ has degree $\leq \beta_v-1$. The number of ways to choose $1$-blocks and to choose the matchings and the indexing modulo $k$ is independent of $N$; the contribution made by power of $w_1$, so we can denote the constant as $C$. Therefore
\begin{align*}
   &\E[{X_{m,N,1}}^l] =\ C\sum_{\substack{2n \leq \ga_1 \leq 4nl\\ 0 \leq i_1 \leq m}}\dots\sum_{\substack{2n \leq \ga_l \leq 4nl\\ 0 \leq i_l \leq m}} \prod_{\nu=1}^lc_{\ga_v} {m \choose i_v} (-1)^{m-i_v}(\frac{w_1N}{k})^{m-i_v}\left(\frac{(\alpha_v+i_v)^{\beta_v}}{\beta_v!}+\tilde{p}(\alpha_v+i_v)\right), \nonumber \\
\end{align*}
which is just $C$ since $\sum_{i=0}^m {m \choose i}(-1)^{m-i}i^m=m!$ and $\sum_{\ga=2n}^{4nl}c_\alpha=1$. This proves the lemma.
\end{proof}

Then following the exactly same steps as the proof of Theorem 5.5 in \cite{BCDHMSTPY}, we can prove the convergence of averaged empirical blip spectral measure:
\begin{thm}\label{thm_as_convergence}
Let $g: \N \rightarrow \N$ be such that there exists an $\delta>0$ for which $g(N) = \omega(N^\delta)$. Then, as $N\to\infty$, the averaged empirical spectral measures $\mu_{N,g,\overline{A},i}$ of the $k$-checkerboard ensemble converge weakly almost-surely to the measure with moments $M_{k,m,i}=\frac{1}{k_i}\mathbb{E}_{k_i}\left[\text{Tr}\left(B^{m_1}\right)\right]$.
\end{thm}


\section{An explicit construction of blips at Fibonacci numbers} \label{sec:fibonacci}
In this section, we give an explicit construction of a sequence of random matrices such that as $N\to\infty$, almost surely there is an eigenvalue, after normalized by dividing $N$, at all Fibonacci numbers. We can apply the same approach to force the normalized blip eigenvalues at any given sequence of real numbers. We begin by extending the definition of the generalized checkerboard matrices to allow $k$ to grow with $N$.
\begin{defi}
For fixed $N\in \mathbb{N}$ and a collection of $k_N \leq N$ real numbers $W_N = (w_1,\dots, w_{k_N})$, the $N\times N$ $(k_N, W_N)$-checkerboard ensemble is the ensemble of matrices $M=(m_{ij})$ given by \begin{align}
    m_{ij} \; = \; \begin{cases}
    a_{ij} & \text{ if } i\not\equiv j\pmod{k_N}, \\
    w_{u} & \text{ if } i\equiv j \equiv u\pmod{k_N},\text{ with  } u\in \{1,2,\dots, k_N\},
    \end{cases}
\end{align} where $a_{ij} = a_{ji}$ are independent and identically distributed random variables with mean $0$, variance $1$, and finite higher moments.
\end{defi}

Let $(k_N)_{N\in\mathbb{N}}$ be a non-decreasing sequence of positive integers with $k_N\leq N$ for each $N\in\mathbb{N}$, $\lim\limits_{N\to\infty}k_N = \infty$ and $k_N = O(\sqrt{N})$. Denote the $n$-th Fibonacci number by $F_n$ where $F_1 = 1$, $F_2 = 2$ and $F_{n+1} = F_{n} + F_{n-1}$.

Let $(A_N)_{N\in \mathbb{N}}$ be a sequence of matrices such that each $A_N$ is a $(k_N, W_N)$-checkerboard matrix with $W_N = (0,\dots, 0)$, and consider the normalized empirical spectral measures \begin{align} \label{defmeas}
\nu_{A_N}(x) \;=\; \frac{1}{N}\sum_{\lambda \text{ \rm an eigenvalue of } A_N}\delta\left(x-\frac{\lambda}{\sqrt{N}}\right).\end{align}
By the same argument as in \S \ref{sec:bulk}, we obtain the following two results.

\begin{proposition}
 Let $(A_N)_{N\in \mathbb{N}}$ be a sequence of matrices such that each $A_N$ is from the $N\times N$ $(k_N, 0)$-checkerboard ensemble. Then the empirical spectral measure $\mu_{A_N}$ defined as \eqref{defmeas} converges almost surely to the Wigner semicircle measure $\sigma_{R}$ with radius \begin{align}
    R = \begin{cases}
    2\sqrt{1-\frac{1}{k}} & \text{ {\rm if} }\lim\limits_{N\to\infty}k_N = k, \\
    2 & \text{ {\rm if} }\lim\limits_{N\to\infty}k_N = \infty.
    \end{cases}
\end{align}
\end{proposition}

\begin{proposition} \label{opnorm}
 Let $(A_N)_{N\in \mathbb{N}}$ be a sequence of matrices such that each $A_N$ is from the $N\times N$ $(k_N, 0)$-checkerboard ensemble. Then as $N\to\infty$, \begin{align}
     \|A_N\|_{op} \;=\; O_{\epsilon}(N^{\frac{1}{2} + \epsilon})
 \end{align} almost surely.
\end{proposition}

For each $N\in\mathbb{N}$, define a fixed $N\times N$ matrix $Z_N$ by \begin{align}
    (Z_N)_{ij} \;=\; \begin{cases}
    F_n & \text{ {\rm if} } i\equiv j \equiv n \pmod{k_N}, \text{ where } 1\leq n\leq k_{N}, \\
    0 & \text{ otherwise}.
    \end{cases}
\end{align}

\begin{lem} \label{fixed}
The matrix $Z_N$ has rank $k_N$, and the $k_N$ nonzero eigenvalues are exactly $F_1\lceil\frac{N}{k_N}\rceil$, $\dots$, $F_{r_N}\lceil\frac{N}{k_N}\rceil$, $F_{r_N+1}\lfloor \frac{N}{k_N}\rfloor$, $\dots$, $F_{k_N}\lfloor \frac{N}{k_N}\rfloor$, where we write $N = q_N k_N + r_N$ with $q_N\in \mathbb{Z}$ and $r_N\in\{0,1,2,\dots, k_N-1\}$.
\end{lem}

\begin{proof}
By definition, the matrix $Z_N$ has at most $k_N$ different columns.

For each $i \in\{1,2,\dots, r_N\}$, define $v_i \in \mathbb{R}^{N}$ by $v_i = \sum_{j = 0}^{q_N} e_{i + jk_N}$, then $v_i$ is an eigenvector of $Z_N$ associated with eigenvalue $F_i\lceil\frac{N}{k_N}\rceil$.

For each $i\in\{r_N+1,\dots, k_N\}$, define $v_i \in \mathbb{R}^{N}$ by $v_i = \sum_{j = 0}^{q_N-1} e_{i + jk_N}$, then $v_i$ is an eigenvector of $Z_N$ associated with eigenvalue $F_i\lfloor\frac{N}{k_N}\rfloor$.
\end{proof}
\begin{rem}
By assumption, we have \begin{align}\lim\limits_{N\to\infty}\frac{\text{{\rm rank} }(Z_N)}{N}\ =\ \lim\limits_{N\to\infty}\frac{k_N}{N} \ =\ 0.\end{align}
\end{rem}

Construct a sequence $(M_N)_{N\in\mathbb{N}} $ of matrices by \begin{align} \label{fibens}
    M_N \; = \;A_N + k_N Z_N.
\end{align}
Note that each $M_N$ is an $N\times N$ checkerboard matrix.
\begin{thm}
Let $(M_N)_{N\in\mathbb{N}}$ be a sequence of checkerboard matrices defined as \eqref{fibens}. Then almost surely $M_N$ has $N-k_N$ eigenvalues of magnitude $O(N^{\frac{1}{2}+\epsilon})$ as $N\to\infty$. Moreover, for all $n\in\mathbb{N}$, almost surely $M_N$ has one eigenvalue of magnitude $NF_n + O(N^{\frac{1}{2}+\epsilon})$ as $N\to\infty$, where $F_n$ denotes the $n$-th Fibonacci number.
\end{thm}
\begin{proof}
Fix any $n\in\mathbb{N}$.
By lemma \ref{fixed}, we know the matrix $k_N Z_N$ has eigenvalue $0$ of multiplicity $N-k_N$, and because $\lim\limits_{N\to\infty} k_N = \infty$, the matrix $k_N Z_N$ has one eigenvalue within the interval $(NF_n - k_N, NF_n + k_N)$ for all sufficiently large $N$. By assumption $\lim\limits_{N\to\infty}\frac{k_N}{N^{1/2}} < \infty$, the matrix $k_N Z_N$ has one eigenvalue of magnitude $NF_n + O(N^{\frac{1}{2}})$ for all sufficiently large $N$.

Let the eigenvalues of $M_N$ and $k_N Z_N$ be arranged in increasing order. As a consequence of Weyl's inequality, we have $|\lambda_{j}(M_N) - \lambda_{j}(k_N Z_N) | \leq \|A_N\|_{op}$ for each $j\in\{1,2,\dots, N\}$.

By Lemma \ref{opnorm}, almost surely $\|A_N\|_{op} = O(N^{\frac{1}{2}+\epsilon})$ as $N\to\infty$.
Therefore, almost surely $M_N$ has $N-k_N$ eigenvalues of magnitude $O(N^{\frac{1}{2}+\epsilon})$, and almost surely $M_N$ has one eigenvalue of magnitude $NF_n + O(N^{\frac{1}{2}+\epsilon})$, as $N\to\infty$.
\end{proof}

Therefore, if normalized by $N$, the limiting spectrum has one eigenvalue at each Fibonacci number. 
For example, Figure \ref{fig:fib} shows a histogram of the normalized eigenvalues with blips at the first $10$ Fibonacci numbers.
\begin{figure}[h]
    \centering
    \includegraphics[width = 16cm, height = 7 cm]{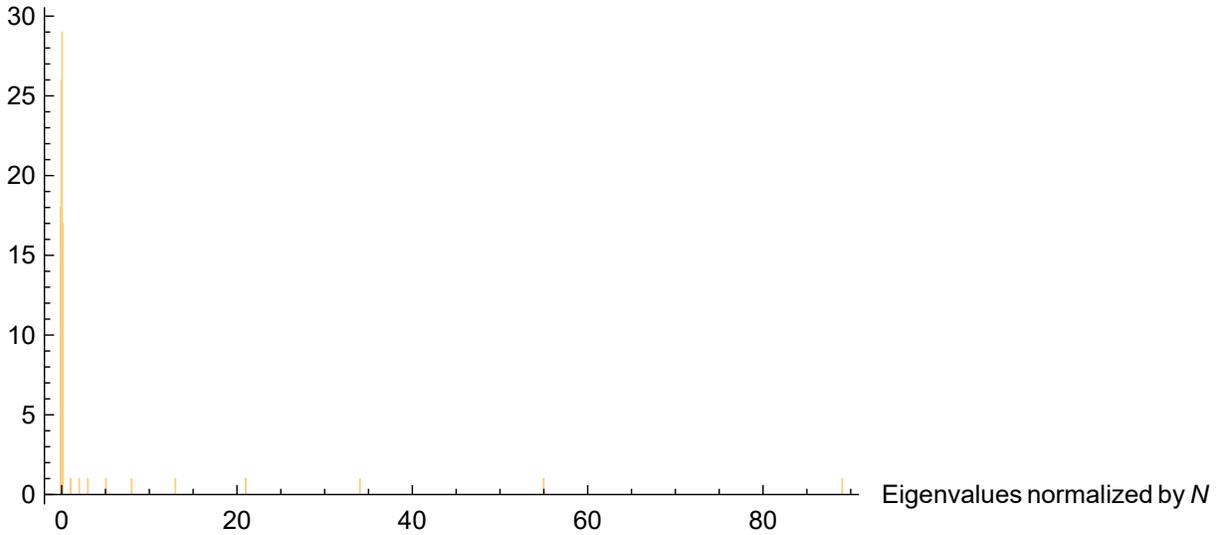}
    \caption{A histogram of the scaled eigenvalue distribution for $M_N$ with $N = 100$ and $k_{100} = 10$.}
    \label{fig:fib}
\end{figure}


\ \\
\end{document}